\documentclass[12pt]{article}
\usepackage{amssymb}
\usepackage{amsthm}
\usepackage{amsmath}
\usepackage{color}
\usepackage{graphics}
\usepackage{epsfig}
\usepackage{hyperref}

\newtheorem{claim}{\bf \t}[part]

\newtheorem{Corollary}{Corollary}[part]
\newtheorem{Definition}{Definition}[part]

\newtheorem{Lemma}{Lemma}[part]
\newtheorem{Proposition}{Proposition}[part]

\newtheorem{Theorem}{Theorem}[part]

\numberwithin{Assumption}{section} \numberwithin{Corollary}{section}
\numberwithin{Definition}{section} \numberwithin{equation}{section}
\numberwithin{Example}{section} \numberwithin{Lemma}{section}
\numberwithin{Proposition}{section} \numberwithin{Remark}{section}
\numberwithin{Theorem}{section}

\def\t{\rho}

\def\text#1{{\rm #1}}
\begin{document}
	\date{}
	\title{\Large \bf On the global well-posedness of 3D inhomogeneous incompressible Navier-Stokes system with density-dependent viscosity}
	\author{\small \textbf{Dongjuan Niu},$^{a,1}$	
		\thanks{The first author wish to thank the the key research project of National Natural Science Foundation (Grant No. 11931010).  E-mail: djniu@cnu.edu.cn}\quad
		\textbf{Lu Wang},$^{a,2}$
		\thanks{E-mail: wlu1130@163.com}}
	
	\maketitle
	\small $^a$ School of Mathematical Sciences, Capital Normal University, Beijing
	100048, P. R. China\\[2mm]

	{\bf Abstract:} In this paper, we are concerned with the global well-posedness of 3D inhomogeneous incompressible Navier-Stokes equations with density-dependent viscosity  when the initial velocity is sufficiently small in the critical Besov space $\dot{B}^{\frac{1}{2}}_{2,1}$. Compared with the previous result of Abidi and Zhang (Science China Mathematics 58 (6) (2015) 1129-1150), we remove the smallness assumption of the viscosity $\mu(\rho_0)-1$ in $L^{\infty}$-norm.


	{\bf Key Words:} Incompressible Navier-Stokes system, Global well-posedness, Time-decay estimates.
	
	
	\section{Introduction } \setcounter{equation}{0}
	\setcounter{Assumption}{0} \setcounter{Theorem}{0}
	\setcounter{Proposition}{0} \setcounter{Corollary}{0}
	\setcounter{Lemma}{0}
	
	\ \ \ The purpose of this paper is to investigate the global well-posedness of the following 3D inhomogeneous incompressible Navier-Stokes system with density-dependent viscosity
	\begin{equation}\label{a1}
	\left\{
	\begin{array}{l}
	\partial_{t}\rho+\operatorname{div}(\rho u)=0,\quad (t,x)\in \mathbb{R}^{+}\times\mathbb{R}^{3},                \\
	\partial_{t}(\rho u)+\operatorname{div}(\rho u\otimes u)-\operatorname{div}(2\mu(\rho)\mathbb{D}u)+\nabla\pi=0, \\
	\operatorname{div} u=0,                                                                                         \\
	(\rho,u)|_{t=0}=(\rho_0,u_0),
	\end{array}
	\right.
	\end{equation}
	where $\rho$ and $u=(u_{1}(t,x),u_{2}(t,x),u_{3}(t,x))$ stand for the density and velocity fields of the fluid respectively, $\mathbb{D}u=\frac{1}{2}(\partial_{i}u_{j}+\partial_j u_i),$ $\pi=\pi(t,x)$ is a scalar pressure function, and in general, the viscosity coefficient $\mu(\rho)$ is a
	smooth, positive function on $[0,\infty)$. Such a system are usually used to describe a fluid that is incompressible but has nonconstant density owing to the complex structure of the flow due to a mixture or pollution. More information on the background of (\ref{a1}) can be referred to \cite{1996L} and so on.
	
	Just similar to the classical Navier-Stokes system, i.e., $\rho=1$ and $\mu(\rho)=\mu>0$, system $(\ref{a1})$ also has the scaling invariance property: if $(\rho, u)$ solves (\ref{a1}) with initial data $(\rho_{0},u_{0})$, then for $\forall~\lambda>0$,
	\begin{equation}\label{aa5}
	(\rho,u)_{\lambda}(t,x)=(\rho(\lambda^{2}t,\lambda x),\lambda u(\lambda^{2}t,\lambda x))
	\end{equation}
	is a solution of (\ref{a1}) with initial data~$(\rho_{0}(\lambda x),\lambda u_{0}(\lambda x)).$ A functional space for data $(\rho_{0},u_{0})$ or for solution $(\rho,u)$ is said to be scaling-invariant of the equation if its norm is invariant under transformation (\ref{aa5}). In particular, $\dot{B}^{\frac{3}{p}}_{p,r}(\mathbb{R}^{3})\times \dot{B}^{\frac{3}{p}-1}_{p,r}(\mathbb{R}^{3})$ is scaling invariant according to (\ref{aa5}).

	For the constant viscosity case, i.e., $\mu(\rho)=\mu>0$, there are many progress about the global well-posedness of the system (\ref{a1}) with initial data belonging to the critical Besov spaces. Owing to the positive lower bound of the initial density, Danchin \cite{2003D} established the global well-posedness of (\ref{a1}) in the N-dimensional case under the smallness conditions  that $\|\rho_0-1\|_{\dot{B}^{\frac{N}{p}}_{p,1}}$ and $\|u_0\|_{\dot{B}^{\frac{N}{p}-1}_{p,1}}$ for any $p\in[1,N]$ with $ N=2,3.$ Later, Abidi-Gui-Zhang \cite{2012AGZ,2013AGZ} removed the size restriction of the initial density and proved the global well-posedness for (\ref{a1}) only provided that $u_{0}$ is small in $\dot{B}^{-1+\frac{3}{p}}_{p,1}(\mathbb{R}^3)$ with $p\in [3,4]$ and they also proved the global existence of 2D Navier-Stokes equations with constant viscosity in \cite{2023AGZ}, assuming that $u_{0}\in L^2\cap\dot{B}^{-1+\frac{2}{p}}_{p,1}(\mathbb{R}^2)$ with $p\in [2,+\infty[$. Recently, Zhang \cite{2020z} proved the global existence of weak solutions when $\rho_0\in L^\infty$ has a positive lower bound and  $u_0$ is sufficiently small in $\dot{B}^{\frac{1}{2}}_{2,1}$.  Danchin-Wang \cite{2022a} verified the uniqueness of solutions constructed by  \cite{2020z}  and proved the global well-poseness under the conditions that $\rho_0$ is close to a positive constant in $L^{\infty}$ space and $u_{0}$ is small in $\dot{B}^{-1+\frac{3}{p}}_{p,1}(\mathbb{R}^3)$ with $p\in(1,3).$ 
	As for the case that $\rho_0$ is allowed to vanish, Craig-Huang-Wang \cite{2013CHW} established the global well-posedness of strong solutions to (\ref{a1}) provided that $u_0$ is small in $\dot{H}^{\frac{1}{2}}.$ More results about well-posedness results in the critical spaces can be refferred to \cite{2011AGZ,2004D,2012DM,2022q,2017j} and so on.
	
	The investigation of the inhomogeneous incompressible Navier-Stokes equaitons with density-dependent viscosity started from Lions \cite{1996L}, who first proved the global existence of weak solutions to system (\ref{a1}) allowing the density to vanish initially. However, the uniqueness and regularities of such weak solutions remain open in two spatial dimensions. Later, Desjardins \cite{1997D} proved the global weak solution with higher regularity for the two-dimensional case provided that the viscosity $\mu(\rho)$ is a small perturbation of a positive constant in $L^{\infty}$-norm. Abidi-Zhang \cite{2015az}  and Gui-Zhang \cite{2009GZ}  generalized the above result to strong solution under the similar assumption of initial density. Recently, when the initial density contains
	vacuum, He-Li-L${\rm \ddot{u}}$ \cite{2021HLL} constructed a unique and global strong solution for (\ref{a1}) under the assumption that
	\begin{equation}\label{aa8}
	\|u_0\|_{\dot{H}^{\beta}}\leq \varepsilon_{0}\quad{\rm with}\quad\beta\in (\frac{1}{2},1],
	\end{equation} 
	where the key idea comes from the important exponential decay-in-time properties of the velocity fields in the whole space $\mathbb{R}^3$. More results can be refferred to \cite{2004CK,2014HW,2015HW,2019LS,2015Z} and so on.

	Compared with the constant viscosity case, it is natural to ask the global well-posedness of system (\ref{a1}) under some small assumptions on $(\rho_0, u_0)$ in the critical spaces when $\mu(\rho)$ depends on the density function $\rho$? Until now, there are only several results related to this topic. For two-dimensional inhomogenous incompressible Navier-Stokes equaitons with density-dependent viscosity, Huang-Paicu-Zhang \cite{2013HPZ} proved that there exists a global and unique solution $(\rho,u)$ satisfying that $\rho-1 \in \mathcal{C}_b([0, \infty) ; \dot{B}_{q, 1}^{\frac{2}{q}}\left(\mathbb{R}^2\right))$ and $u \in \mathcal{C}_b([0, \infty) ; \dot{B}_{p, 1}^{-1+\frac{2}{p}}\left(\mathbb{R}^2\right)) \cap L^1(\mathbb{R}^{+} ; \dot{B}_{p, 1}^{1+\frac{2}{p}}\left(\mathbb{R}^2\right))$ for $1<q\leq p<4$ provided that the fluctuation of the initial density is sufficiently small. If $1<p<2N,$ $ 0<\underline{\mu}<\mu(\rho),$ Abidi \cite{2007A} proved that N-dimensional inhomogenous incompressible Navier-Stokes equaitons with density-dependent viscosity has a global solution provided that 
	\begin{equation}
	\|\rho_0-1\|_{\dot{B}^{\frac{N}{p}}_{p,1}}+\|u_0\|_{\dot{B}^{\frac{N}{p}-1}_{p,1}}\leq c_0,
	\end{equation}
	where $c_0$ sufficiently small. When the initial density is strictly away from vacuum, Abidi-Zhang \cite{2015AZ} obtained the global well-posedness of 3D system (\ref{a1}) under 
	\begin{equation}\label{aa1}
	\|\mu(\rho_0)-1\|_{L^\infty}\leq \varepsilon_{0} \quad {\rm and}\quad \|u_0\|_{L^2}\|\nabla u_0\|_{L^2}\leq \varepsilon_1,
	\end{equation}
	where $\varepsilon_0, \varepsilon_1$ are sufficently small constant.
	
	In this paper, we intend to investigate the global well-posedness of (\ref{a1}) only provided that initial velocity $u_0$ is sufficiently small in the critical Besov space $\dot{B}^{\frac{1}{2}}_{2,1}$? In other word, compared to the previous result of \cite{2015AZ}, we expect to remove the smallness assumptions on $\|\mu(\rho_0)-1\|_{L^\infty}$ in (\ref{aa1}), which plays the key role to derive the desired {\em a priori} estimates of $\|\nabla u\|_{L^1(L^{\infty})}$. We are in the position to derive that  the bound of  $\|\nabla u\|_{L^\infty(L^2)}$ can be controlled by $\|u_0\|_{ \dot{B}^{\frac{1}{2}}_{2,1}}$. 
	However, the estimate of $\|\nabla u\|_{L^\infty(L^2)}$ is strongly coupled with the decay estimates of velocity fields $u$, which in turn bring some new difficulties. The new ingredient here is to decompose the velocity fields $u$ into $v$ and $w$, where $v$ satisfies the 3D classical Navier-Stokes equations
	\begin{equation}\label{a11}
	\left\{\begin{array}{l}
	\partial_t v+v \cdot \nabla v-\Delta v+\nabla \pi_v=0, \\
	\operatorname{div} v=0,                                \\
	\left. v\right|_{t=t_0}=u\left(t_0\right),
	\end{array}\right.
	\end{equation}
	and the perturbation $w:=u-v$ satisfies
	\begin{equation}\label{a12}
	\left\{\begin{array}{l}
	\partial_t \rho+\operatorname{div}(\rho(v+w))=0,                                                                         \\
	\rho \partial_t w+\rho(v+w) \cdot \nabla w-\operatorname{div}(2\mu(\rho)\mathbb{D}w)+\nabla \pi_w                        \\
	\quad=(1-\rho)\left(\partial_t v+v \cdot \nabla v\right)-\rho w \cdot \nabla v+\operatorname{div}(2(\mu(\rho)-1)\mathbb{D}v), \\
	\operatorname{div} w=0,                                                                                                  \\
	\left.\rho\right|_{t=t_0}=\rho\left(t_0\right),\left.\quad w\right|_{t=t_0}=0.
	\end{array}\right.
	\end{equation}
	In addition, we mention that the zero intial velocity of system \eqref{a12} is very essensional. More details can be found later in Lemma \ref{Lemma-5.3}.

	Throughout this paper, we shall always assume that $\rho_0$ having a positive lower bound and
	\begin{equation} \label{a6}
	0<\underline{\mu}\leq\mu(\rho)\leq \overline{\mu}, \quad \mu(\cdot)\in\mathrm{W}^{3,\infty}(\mathbb{R}^{+})\quad{\rm  and }\quad\mu(1)=1.
	\end{equation}
	\par Now, we state our main result:
	\begin{Theorem}\label{Theorem-1} Let $q\in(3,6)$ and $\delta\in(\frac{1}{2},\frac{3}{4})$. Assume that the initial data $(\rho_{0},u_{0})$ satisfies the regularity condition
		$$\rho_{0}-1\in B^{\frac{3}{2}}_{2,1}(\mathbb{R}^{3}),~\nabla\mu(\rho_0)\in L^{q},~u_{0}\in \dot{B}^{\frac{1}{2}}_{2,1}\cap\dot{H}^{-2\delta}(\mathbb{R}^{3}).$$
		Then there exists some small positive constant $\varepsilon$ depending on $\|\rho_{0}-1\|_{B^{\frac{3}{2}}_{2,1}},$ $\|\nabla\mu(\rho_0)\|_{L^{q}}$ and $\|u_0\|_{\dot{H}^{-2\delta}}$ such that if
		\begin{equation}\label{a7}
		\|u_{0}\|_{\dot{B}^{\frac{1}{2}}_{2,1}}\leq\varepsilon,
		\end{equation}
		the Cauchy problem $(\ref{a1})$ admits a unique global strong solution~$(\rho,u,\nabla\pi)$ satisfying that for any $0<T<\infty$,
		\begin{equation}
		\left\{\begin{array}{l}
		\rho-1 \in \mathcal{C}\left([0,T] ;B^{\frac{3}{2}}_{2,1}\right),~\nabla \mu(\rho) \in \mathcal{C}\left([0,T] ; L^q\right),                \\
		u \in \mathcal{C}\left([0,T] ; \dot{B}^{\frac{1}{2}}_{2,1}\right)\cap L^{1}\left([0,T]; \dot{B}^{\frac{5}{2}}_{2,1}\right),               \\
		\pi\in L^{1}\left([0,T] ; L^2\cap\dot{B}^{\frac{3}{2}}_{2,1}\right)~and~u_t \in L^{\infty}\left([0,T];\dot{B}^{\frac{1}{2}}_{2,1}\right), \\
		\end{array}\right.
		\end{equation}
	\end{Theorem}
	\vskip 0.5cm
	\par This paper is organized as follows. Section 2, we recall some basic Littlewood-Paley theory, as well as some necessary linear estimates of Navier-Stokes equations. Section 3, we establish some short time estimates related to the system (\ref{a1}). Proof of Theorem \ref{Theorem-1} is completed in Section 4, where we obtain the global well-posedness of (\ref{a1}) under the smallness assumptions (\ref{a7}).\\
	
	\par Let us complete this section with the notations we are going to use in this context.\\
	
	\par \textbf{Notations}~Let $A, B$ be two operators, we denote $[A, B]=AB-BA,$ the commutator between $A$ and $B.$ For $a\lesssim b$, we mean that there is a uniform constant $C$, which may be different on different lines, such that $a\leq Cb$. We shall denote by $(a~|~b)_{L^2}$ the $L^2(\mathbb{R}^3)$ inner product of $a$ and $b.$ For $X$ a Banach space, we denote $\|(f,g)\|_{X}\stackrel{\Delta}{=}\|f\|_{X}+\|g\|_{X}.$ Finally, $(c_{j})_{j\in\mathbb{Z}}$ (resp. $(d_{j})_{j\in\mathbb{Z}}$) will be a generic element of $\ell^{2}(\mathbb{Z})$ (resp. $\ell^{1}(\mathbb{Z})$) so that $\sum_{j\in\mathbb{Z}}c^{2}_{j}=1$ (resp. $\sum_{j\in\mathbb{Z}}d_{j}=1$).
	
	\section{Preliminaries}
	In this section, we introduce some notations and conventions, and recall some standard Lemmas of Besov space, which will be used throughout this paper, as well as some necessary linear estimates of Navier-Stokes equations.
	
	\subsection{Littlewood-Paley theory}
	Since the proof of Theorem requires a dyadic decomposition of the Fourier variables, we shall recall the Littlewood-Paley decomposition. See \cite{2011BCD} for more details.
	
	\begin{Definition}\label{Definition-1} (\cite{2011BCD}) Let $(p,r)\in[1,\infty]^{2},~s\in\mathbb{R}$ and $u\in\mathcal{S}'(\mathbb{R}^{3}),$ we set
		\begin{equation}\label{a13}
		\|u\|_{B^{s}_{p,r}}\stackrel{\mathrm{ def }}{=}(2^{qs}\|\Delta_{q}u\|_{L^{p}})_{l^{r}},\quad \|u\|_{\dot {B}^{s}_{p,r}}\stackrel{\mathrm{ def }}{=}(2^{qs}\|\dot{\Delta}_{q}u\|_{L^{p}})_{l^{r}}.
		\end{equation}
	\end{Definition}
	
	\begin{Definition}\label{Definition-2} (\cite{2011BCD}) Let $(p,\lambda,r)\in[1,\infty]^{3},$ $s\in\mathbb{R},$ $T\in(0,+\infty],$ and $u\in\mathcal{S}'(\mathbb{R}^{3}),$ we set
		\begin{equation}\label{a14}
		\|u\|_{\widetilde{L}^{\lambda}_{T}(B^{s}_{p,r})}\stackrel{\mathrm{ def }}{=}(2^{qs}\|\Delta_{q}u\|_{L^{\lambda}_{T}(L^{p})})_{l^{r}},\  \|u\|_{\widetilde{L}^{\lambda}_{T}(\dot{B}^{s}_{p,r})}\stackrel{\mathrm{ def }}{=}(2^{qs}\|\dot{\Delta}_{q}u\|_{L^{\lambda}_{T}(L^{p})})_{l^{r}}.
		\end{equation}
	\end{Definition}
	
	\begin{Lemma}\label{Lemma-2.1} $(\cite{2011BCD})$ Let $\mathcal{C}$ be the annulus $\{\xi\in\mathbb{R}^{N}:3/4\leq|\xi|\leq 8/3\}.$ There exist radial functions $\chi$ and $\varphi$, valued in the interval $[0,1],$ belonging respectively to $\mathcal{D}(\mathcal{C})$, and such that
		\begin{equation}\label{a15}
		\begin{aligned}
		& \forall~\xi\in\mathbb{R}^{N},~\chi(\xi)+ \sum_{j\geq 0} {\varphi(2^{-j\xi})}=1,\\
		& \forall~\xi\in\mathbb{R}^{N},~ \sum_{j\in\mathbb{Z}}{ \varphi(2^{-j\xi})}=1,\\
		& |j-j'|\geq 2\Longrightarrow {\rm Supp}\varphi(2^{-j\xi\cdot})\cap{\rm Supp}\varphi(2^{-j'\xi\cdot})=\varnothing,\\
		& j\geq 1\Longrightarrow {\rm Supp}\chi\cap{\rm Supp}\varphi(2^{-j'\xi\cdot})=\varnothing.
		\end{aligned}
		\end{equation}
		the set $\widetilde{\mathcal{C}}=B(0,2/3)+\mathcal{C}$~is an annulus, and we have
		\begin{equation}\label{a16}
		|j-j'|\geq 5\Longrightarrow 2^{j'}\widetilde{\mathcal{C}}\cap 2^{j}\mathcal{C}=\varnothing.
		\end{equation}
	\end{Lemma}
	
	\begin{Lemma}\label{Lemma-2.2} $(\cite{2011BCD})$ 
		Let $\mathcal{C}$ be an annulus and $\mathcal{B}$ a ball. There exists a constant $C>0$ such that for any nonnegative integer $k,$ any couple  $(p,q)$ in $[1,\infty]^{2}$ with $q\geq p\geq 1$ and any $u$ of $L^{p}$ satisfying
		\begin{enumerate}
			\item If ${\rm Supp} \ \hat{u}\subset\lambda\mathcal{B},$ we have
			\begin{equation}\label{a17}
			\sup_{|\alpha|=k}\|\partial^{\alpha}u\|_{L^{q}}\leq C^{k+1}\lambda^{k+N(\frac{1}{p}-\frac{1}{q})}\|u\|_{L^{p}}.\end{equation}
			\item If ${\rm Supp} \ \hat{u}\subset\lambda\mathcal{C},$ we have
			\begin{equation}\label{a18} C^{-k-1}\lambda^{k}\|u\|_{L^{p}}\leq \|D^{k}u\|_{L^{p}}\leq C^{k+1}\lambda^{k}\|u\|_{L^{p}}.
			\end{equation}
		\end{enumerate} 
	\end{Lemma}
	
	\par In the rest of the paper, we shall frequently use homogeneous Bony's decomposition:
	\begin{equation}\label{a19}
	uv=T_{u} v+T'_{v}u=T_{u} v+T_{v} u+R(u,v),
	\end{equation}
	where
	\begin{equation}\label{a20}
	\begin{aligned}
	& T_{u}v= \sum_{q\in\mathbb{Z}}\dot{S}_{q-1}u\dot{\Delta}_{q}v,\quad T'_{v}u=\sum_{q\in\mathbb{Z}}\dot{S}_{q+2}v\dot{\Delta}_{q}u,\\
	& R(u,v)=\sum_{q\in\mathbb{Z}}\dot{\Delta}_{q}u\widetilde{\dot{\Delta}}_{q}v,\quad \widetilde{\dot{\Delta}}_{q}v=\sum_{|q'-q|\leq1}\dot{\Delta}_{q'}v.
	\end{aligned}
	\end{equation}
	Similarly, we can obtain inhomogeneous Bony's decomposition \cite{2011BCD}.\\
	\par Then we shall derive some conclusions about the transport equation, such as the commutator's estimates, which will be frequently used throughout the succeeding sections. The proof process can be referred to \cite{2011BCD} .\\
	
	\begin{Lemma}\label{Lemma-2.3} $(\cite{2011BCD})$
		Let $r\in[1,\infty],~p\in[1,\infty],$ $-1\leq s\leq 1$ and $\operatorname{div}v=0.$ Then there holds
		\begin{enumerate}
			\item If $s=-1,$
			\begin{equation}\label{a21}
			\sup_{q}2^{-q}\|[v\cdot\nabla;\dot{\Delta}_{q}]f\|_{L^{1}_{t}(L^{p})}\lesssim \int^{t}_{0}\|\nabla v\|_{\dot{B}^{0}_{\infty,1}}\|f\|_{\dot{B}^{-1}_{p,\infty}}dt;\end{equation}
			\item If $-1<s<1,$ 
			\begin{equation}\label{a22}
			(\sum_{q}2^{qsr}\|[v\cdot\nabla;\dot{\Delta}_{q}]f\|^{r}_{L^{1}_{t}(L^{p})})^{\frac{1}{r}}\lesssim \int^{t}_{0}\|\nabla v\|_{L^{\infty}}\|f\|_{\dot{B}^{s}_{p,r}}dt;\end{equation}
			\item If $s=1,$ 
			\begin{equation}\label{a23}
			\sum_{q}2^{q}\|[v\cdot\nabla;\dot{\Delta}_{q}]f\|_{L^{1}_{t}(L^{p})}\lesssim \int^{t}_{0}\|\nabla v\|_{\dot{B}^{0}_{\infty,1}}\|f\|_{\dot{B}^{1}_{p,1}}dt.\end{equation}
		\end{enumerate}
	\end{Lemma}
	\begin{Lemma}\label{Lemma-2.4} $(\cite{2013AGZ})$
		Let $p\in[2,6]$, $r\in[1,\infty]$ and if $p=6$, we have $r=1$. Assume that $a_{0}\in\dot{B}^{\frac{3}{2}}_{2,1}(\mathbb{R}^{3}) $ and $\nabla u\in L^{1}_{T}(\dot{B}^{\frac{3}{p}}_{p,r}\cap L^\infty)$ with $\operatorname{div}u=0,$ the function $a\in\mathcal{C}([0,T];\dot{B}^{\frac{3}{2}}_{2,1}(\mathbb{R}^{3}))$ solves
		\begin{align}\label{a24}
		\bigg\{\begin{array}{l}
		\partial_{t}a+u\cdot\nabla a=0,\quad (t,x)\in[0,T]\times\mathbb{R}^{3},\\
		a|_{t=0}=a_{0}.
		\end{array}\bigg.
		\end{align}
		Then there holds that for $\forall~t\in(0,T],$ 
		\begin{align}\label{a25}
		\|a(t)\|_{\dot{B}^{\frac{3}{2}}_{2,1}}\leq \|a_{0}\|_{\dot{B}^{\frac{3}{2}}_{2,1}} \exp\bigg\{C\|\nabla u\|_{L^{1}_{t}(\dot{B}^{\frac{3}{p}}_{p,r}\cap L^\infty)}\bigg\}.
		\end{align}
	\end{Lemma}
	
	\vskip 0.5cm
	
	\subsection{Regularity results}
	\par The following regularity results on the Stokes system will be useful for our derivation of higher order \emph{a priori} estimates:
	
	\begin{Lemma}\label{Lemma-5.1}
		For positive constants $\underline{\mu}, \bar{\mu},$ and $q\in(3,\infty),$ in addition, assume that $\mu(\rho)$ satisfies 
		\begin{equation}\label{d22}
		\nabla\mu(\rho)\in L^{q},\quad 0<\underline{\mu}\leq\mu(\rho)\leq\bar{\mu}<\infty.
		\end{equation}
		Then, if $F\in L^{r}$ with $r\in[\frac{2q}{q+2},q],$ there exists some positive $C$ depending only on $\underline{\mu},\bar{\mu},q$ and $r$ such that the unique weak solution $(u,\pi)\in H^{1}\times L^{2}$ to the Cauchy problem
		\begin{equation}\label{d23}
		\left\{\begin{array}{l}
		-\operatorname{div} (2\mu(\rho)\mathbb{D}u)+\nabla\pi=F,\ \ \ \ \ \ \ \ \ \ \ \ \  (t,x)\in\mathbb{R}_{+}\times\mathbb{R}^{3}, \\
		\operatorname{div} u=0, \ \ \ \ \ \ \ \ \ \ \ \ \ \ \ \ \ \ \  \ \ \ \ \ \ \ \ \ \ \ \ \ \ \ (t,x)\in\mathbb{R}_{+}\times\mathbb{R}^{3},\\
		u(x)\longrightarrow 0,\ \ \ \ \ \ \ \ \ \  \ \ \ \ \  \ \ \ \ \ \ \ \ \ \ \ \ \ \ \ \ \ \ |x|\longrightarrow 0
		\end{array}\right.
		\end{equation}	
		satisfies
		\begin{equation}\label{d24}
		\|\nabla^{2}u\|_{L^{2}}\leq \|F\|_{L^{2}}+\|\nabla \mu(\rho)\|^{\frac{q}{q-3}}_{L^{q}}\|\nabla u\|_{L^{2}},
		\end{equation}
		and
		\begin{equation}\label{d25}
		\|\nabla^{2} u\|_{L^{r}}\lesssim\|F\|_{L^{r}}+\|\nabla\mu(\rho)\|^{\frac{q(5r-6)}{2r(q-3)}}_{L^{q}}\{\|\nabla u\|_{L^{2}}+\|(-\Delta)^{-1}\operatorname{div}F\|_{L^{2}}\}.
		\end{equation}
	\end{Lemma}
	
	\begin{proof}
		The first equation of (\ref{d23}) can be re-written as
		\begin{equation}\label{d27}
		\mu(\rho)\Delta u= -2\mathbb{D} u\cdot\nabla\mu(\rho)+\nabla\pi-F.
		\end{equation}
		Applying $\dot{\Delta}_{j} \mathbb{P}$ to (\ref{d27}) and using a standard commutator process gives
		$$
		\dot{\Delta}_{j}\Delta u=\frac{[\mu(\rho);\dot{\Delta}_{j}\mathbb{P}]\Delta u}{\mu(\rho)}-\frac{\dot{\Delta}_{j}\mathbb{P}\{2\mathbb{D} u\cdot\nabla \mu(\rho)\}}{\mu(\rho)}-\frac{\dot{\Delta}_{j}\mathbb{P}F}{\mu(\rho)}.
		$$ 
		By virtue of $0<\underline{\mu}\leq\mu(\rho)$, it is easy to get
		\begin{equation}\label{d28}
		\begin{aligned}
		\|\dot{\Delta}_{j}\nabla^{2} u\|_{L^{2}}\lesssim \|[\mu(\rho);\dot{\Delta}_{j}\mathbb{P}]\Delta u\|_{L^{2}}+\|\dot{\Delta}_{j}\{2\mathbb{D} u\cdot\nabla \mu(\rho)\}\|_{L^{2}}+\|\dot{\Delta}_{j}F\|_{L^{2}}.
		\end{aligned}
		\end{equation}	
		Using Bony's decomposition, one has
		$$
		[\mu(\rho);\dot{\Delta}_{j}\mathbb{P}]\Delta u= [T_{\mu(\rho)};\dot{\Delta}_{j}\mathbb{P}]\Delta u+T^{'}_{\dot{\Delta}_{j} \Delta u}\mu(\rho)-\dot{\Delta}_{j}\mathbb{P}(T_{\Delta u}\mu(\rho))-\dot{\Delta}_{j}\mathbb{P}\mathcal{R}(\mu(\rho),\Delta u).
		$$
		It follows again from the above estimate,~which implies that
		$$
		\begin{aligned}
		\|[T_{\mu(\rho)};\dot{\Delta}_{j}\mathbb{P}]\Delta u\|_{L^{2}}
		&\lesssim\sum_{|j-k|\leq 4}2^{-j}\|\nabla\dot{S}_{k-1}\mu(\rho)\|_{L^{q}}\|\dot{\Delta}_{k}\Delta u\|_{L^{q^{*}}}\\
		& \lesssim c_{j}\|\nabla \mu(\rho)\|_{L^{q}}\|\nabla u\|_{\dot{B}^{0}_{q^{*},2}},
		\end{aligned}
		$$
		where $\frac{1}{q}+\frac{1}{q^{*}}=\frac{1}{2}.$ Notice that
		$$
		\begin{aligned}
		\|T^{'}_{\dot{\Delta}_{j}\Delta u}\mu(\rho)\|_{L^{2}}&\lesssim\sum_{k\geq j-2}\|\dot{S}_{k+2}\dot{\Delta}_{j} \Delta u\|_{L^{q^{*}}}\|\dot{\Delta}_{k} \mu(\rho)\|_{L^{q}}\\&
		\lesssim \|\dot{\Delta}_{j} \nabla u\|_{L^{q^{*}}}\sum_{k\geq j-2}2^{j-k}2^{k}\|\dot{\Delta}_{k}\mu(\rho)\|_{L^{q}}\\&
		\lesssim c_{j}\|\nabla u\|_{\dot{B}^{0}_{q^{*},2}}\|\nabla\mu(\rho)\|_{L^{q}}.
		\end{aligned}
		$$
		Similarly, we have 
		$$
		\begin{aligned}
		\|\dot{\Delta}_{j}\mathbb{P}(T_{\Delta u} \mu(\rho))\|_{L^{2}}&\lesssim\sum_{|k-j|\leq 4}\|\dot{S}_{k-1}\Delta u\|_{L^{q^{*}}}\|\dot{\Delta}_{k}\mu(\rho)\|_{L^{q}}\\&
		\lesssim \sum_{|k-j|\leq 4}2^{-k}\|\dot{S}_{k-1}\Delta u\|_{L^{q^{*}}}2^{k}\|\dot{\Delta}_{k}\mu(\rho)\|_{L^{q}}\\&
		\lesssim c_{j}\|\nabla u\|_{\dot{B}^{0}_{q^{*},2}}\|\nabla\mu(\rho)\|_{L^{q}}.
		\end{aligned}
		$$
		The same estimate holds for $\dot{\Delta}_{j}\mathbb{P}\mathcal{R}(\mu(\rho),\Delta u),$ applying Lemma \ref{Lemma-2.2} gives
		$$
		\begin{aligned}
		\|\dot{\Delta}_{j}\mathbb{P}\mathcal{R}(\mu(\rho),\Delta u)\|_{L^{2}}&\lesssim\sum_{k\geq j-2}2^{\frac{3}{q}j}\|\dot{\Delta}_{j}\mathbb{P}\mathcal{R}(\mu(\rho),\Delta u)\|_{L^{\frac{2q}{q+2}}}\\&
		\lesssim\sum_{k\geq j-2}2^{(j-k)\frac{3}{q}}2^{(\frac{3}{q}-1)k}\|\dot{\Delta}_{k}\Delta u\|_{L^{2}}\|\dot{\Delta}_{k}\nabla \mu(\rho)\|_{L^{q}}\\&
		\lesssim c_{j}\| \nabla u\|_{\dot{B}^{\frac{3}{q}}_{2,2}}\|\nabla\mu(\rho)\|_{L^{q}}.
		\end{aligned}
		$$
		By virtue of the Gagliardo-Nirenberg inequality $\|\nabla u\|_{\dot{B}^{\frac{3}{q}}_{2,2}}\lesssim\|\nabla u\|^{1-\frac{3}{q}}_{L^{2}}$ $\|\nabla^{2}u\|^{\frac{3}{q}}_{L^{2}}$ and $\dot{B}^{\frac{3}{q}}_{2,2}\hookrightarrow \dot{B}^{0}_{q^{*},2}\hookrightarrow L^{q^{*}}$ with $q^{*}\geq 2$, one has
		\begin{align}\label{d29}
		\|[\mu(\rho),\dot{\Delta}_{j}\mathbb{P}]\Delta u\|_{L^{2}}\lesssim c_{j}\|\nabla u\|^{1-\frac{3}{q}}_{L^{2}}\|\nabla^{2}u\|^{\frac{3}{q}}_{L^{2}}\|\nabla\mu(\rho)\|_{L^{q}}.
		\end{align}
		The same estimate hold for $\mathbb{D} u\cdot\nabla \mu(\rho).$ Substituting (\ref{d29}) into (\ref{d28}) and by Young's inequality, we obtain
		\begin{equation}\label{d30}
		\begin{aligned}
		\|\nabla^{2}u\|_{L^{2}}&\leq C(\sum_{j\in\mathbb{Z}}\|[\mu(\rho);\dot{\Delta}_{j}\mathbb{P}]\Delta u\|^{2}_{L^{2}})^{\frac{1}{2}}+C\|\mathbb{D} u\cdot\nabla \mu(\rho)\|_{L^{2}}+C\|F\|_{L^{2}}\\
		&\leq C\|\nabla u\|^{1-\frac{3}{q}}_{L^{2}}\|\nabla^{2}u\|^{\frac{3}{q}}_{L^{2}}\|\nabla\mu(\rho)\|_{L^{q}}+C\|F\|_{L^{2}}\\&
		\leq \frac{1}{2}\|\nabla^{2}u\|_{L^{2}}+C\|\nabla u\|_{L^{2}}\|\nabla\mu(\rho)\|^{\frac{q}{q-3}}_{L^{q}}+C\|F\|_{L^{2}}.
		\end{aligned}
		\end{equation}
		\par Futhermore, it follows from $(\ref{d23})_{1}$ that
		\begin{equation}\label{d31}
		\pi=-(-\Delta)^{-1}\operatorname{div}F-(-\Delta)^{-1}\operatorname{div}\operatorname{div}(2\mu(\rho)\mathbb{D}u),
		\end{equation}
		which gives
		\begin{equation}\label{dd24}
		\|\pi\|_{L^{2}}\lesssim \|(-\Delta)^{-1}\operatorname{div}F\|_{L^{2}}+\|\nabla u\|_{L^{2}}.
		\end{equation}
		Then we rewrite (\ref{d27}) as
		\begin{equation}\label{d32}
		-\Delta u+\nabla (\frac{\pi}{\mu(\rho)})=\frac{F}{\mu(\rho)}+\frac{2\mathbb{D}u\cdot\nabla \mu(\rho)}{\mu(\rho)}-\frac{\pi\nabla\mu(\rho)}{\mu(\rho)^{2}}.
		\end{equation}
		Applying the standard $L^{r}$-estimates to the Stokes system (\ref{d32}) yields that, for 
		\begin{equation}
		\begin{aligned}
		\|\nabla^{2} u\|_{L^{r}}+\left\|\nabla (\frac{\pi}{\mu(\rho)}) \right\|_{L^{r}}
		\leq C\left\|\frac{F}{\mu(\rho)}\right\|_{L^{r}}+C\left\|\frac{2\mathbb{D}u\cdot\nabla \mu(\rho)}{\mu(\rho)}\right\|_{L^{r}}+C\left\|\frac{\pi\nabla\mu(\rho)}{\mu(\rho)^{2}}\right\|_{L^{r}},
		\end{aligned}
		\end{equation}
		from which, we deduce that
		\begin{equation}\label{d33}
		\begin{aligned}
		\|\nabla^{2} u\|_{L^{r}}+\left\|\nabla (\frac{\pi}{\mu(\rho)}) \right\|_{L^{r}}
		\leq& C\|F\|_{L^{r}}+C\|2\mathbb{D}u\cdot\nabla \mu(\rho)\|_{L^{r}}+C\|\pi\nabla\mu(\rho)\|_{L^{r}}\\
		\leq& C\|F\|_{L^{r}}+C\|\nabla u\|^{\alpha}_{L^{2}}\|\nabla^{2}u\|^{1-\alpha}_{L^{r}}\|\nabla \mu(\rho)\|_{L^{q}}\\ &+C\|\pi\|^{\alpha}_{L^{2}}\left\|\nabla(\frac{\pi}{\mu(\rho)})\right\|^{1-\alpha}_{L^{r}}\|\nabla\mu(\rho)\|_{L^{q}},
		\end{aligned}
		\end{equation}
		where $\alpha$ satisfies
		$$\frac{1}{r}-\frac{1}{q}=\frac{\alpha}{2}+(1-\alpha)(\frac{1}{r}-\frac{1}{3}), \quad {\rm or}~ \alpha=\frac{2r}{5r-6}\cdot\frac{q-3}{q}.$$
		By Young's inequality and (\ref{dd24}),
		\begin{equation}\label{d34}
		\|\nabla^{2} u\|_{L^{r}}\lesssim\|F\|_{L^{r}}+\|\nabla\mu(\rho)\|^{\frac{q(5r-6)}{2r(q-3)}}_{L^{q}}\{\|\nabla u\|_{L^{2}}+\|(-\Delta)^{-1}\operatorname{div}F\|_{L^{2}}\}.
		\end{equation}
		The proof of Lemma \ref{Lemma-5.1} is finished.
	\end{proof}
	
	\section{Local well-posedness of (\ref{a1})}
	If the density $\rho_0$ is away from zero, we denote by $a\stackrel{\text{def}}{=}\rho^{-1}-1$ and $\widetilde{\mu}(a)\stackrel{\text{def}}{=}\mu(\rho)$, then the system (\ref{a1}) can be equivalently reformulated as
	\begin{equation} \label{a2}
	\left\{\begin{array}{l}
	\partial_{t}a+u\cdot\nabla a=0,\quad (t,x)\in \mathbb{R}^{+}\times\mathbb{R}^{3},\\
	\partial_{t}u+u\cdot\nabla u-{\rm  div }(2(1+b)\mathbb{D} u)+(1+a)\nabla\pi=-\nabla\lambda~\mathbb{D}u,\\
	\operatorname{div} u=0,\\
	(a,u)|_{t=0}=(a_{0},u_{0}).
	\end{array}\right.
	\end{equation}
	For notational simiplicity, we denote by
	\begin{equation} \label{a3}
	b \stackrel{\text { def }}{=}(1+a) \tilde{\mu} (a)-1 \quad \text { and } \quad \lambda \stackrel{\text { def }}{=}2 \int_0^{a} \tilde{\mu}(s)ds.
	\end{equation} 
	\par We start with the local existence of strong solutions which has been proved in Theorem 3.1 of \cite{2022q}:
	\begin{Lemma}\label{Lemma-2} $(\cite{2022q})$
		Let $a_0 \in B_{2, 1}^{\frac{3}{2}}\left(\mathbb{R}^3\right)$ and $ u_0 \in \dot{B}_{2, 1}^{\frac{1}{2}}\left(\mathbb{R}^3\right)$ with $\operatorname{div} u_0=0$. There exits a small constant $\varepsilon_{0}$ depending on $\|a_0\|_{B_{2, 1}^{\frac{3}{2}}}$ so that if
		\begin{equation}\label{aa12}
		\left\|u_0\right\|_{\dot{B}_{2, 1}^{\frac{1}{2}}} \leq \varepsilon_{0},
		\end{equation}
		then there is a positive time $T^*>1$ such that (\ref{a2}) has a unique local-in-time solution $a\in \mathcal{C}_b([0, T^*] ;$ $ B_{2, 1}^{\frac{3}{2}}(\mathbb{R}^3))$, $ u\in\mathcal{C}_b([0, T^*] ; \dot{B}_{2, 1}^{\frac{1}{2}}(\mathbb{R}^3))$ $\cap $ $L^1([0, T^*] ;$ $ \dot{B}_{2, 1}^{\frac{5}{2}}(\mathbb{R}^3))$
		and $u_t,\nabla\pi\in L^1([0, T^*] ; \dot{B}_{2,1}^{\frac{1}{2}}(\mathbb{R}^3)).$
	\end{Lemma}
	\vskip 0.5cm
	
	\par In comparison with the local solution in Lemma \ref{Lemma-2}, we need to prove the propagation of the regularity of $u$ in $\dot{H}^{-2\delta}(\mathbb{R}^{3})$ with $\delta\in]1/2,3/4[$ before showing that $T^*$ can be extended to $+\infty$ for system (\ref{a1}), which guarantees the decay estimates of the velocity field $u$ at infinity.
	
	\begin{Proposition}\label{Proposition-4.3} Let $u_{0}\in \dot{H}^{-2\delta}(\mathbb{R}^{3})$ with $\delta\in(\frac{1}{2},\frac{3}{4})$ and $a,$ $b,$ $\lambda\in \widetilde{L}^{\infty}_{T}(B^{\frac{3}{2}}_{2,1})$ satisfy
		\begin{align}
		1+a\geq \underline{a}>0\quad{\rm and}\quad 1+b\geq \underline{b}>0.
		\end{align}
		We assume that $u$, $v$ be two solenoidal vector field which satisfy $\operatorname{div} u=\operatorname{div} v=0$ such that $(u,v)\in C([0, T];\dot{B}^{0}_{3,1}(\mathbb{R}^{3}))$ $\cap$ $L^{1}_{loc}([0,T];\dot{B}^{2}_{3,1}(\mathbb{R}^{3}))$ solve
		\begin{align}\label{c54}
		\left\{\begin{array}{lll}
		\partial_{t} u+v\cdot\nabla u-\operatorname{div} (2(1+b)\mathbb{D}u)+(1+a)\nabla\pi=-\nabla\lambda~\mathbb{D}u,\\
		\operatorname{div} u=0,\\
		u|_{t=0}=u_{0}.
		\end{array}\right.
		\end{align}
		For any small positive constant $c$, there exists a sufficiently large integer $k\in\mathbb{Z},$ such that
		\begin{align}\label{c55}
		\|(a-\dot{S}_{k}a,~b-\dot{S}_{k}b,~\lambda-\dot{S}_{k}\lambda)\|_{\widetilde{L}^{\infty}_{T}(B^{\frac{3}{2}}_{2,1})}\leq c,
		\end{align}
		then there exists a constant $C,$ which depends only on $c$ and for any $0<t<T$, satisfies
		\begin{align}\label{c56}
		\begin{aligned}
		&\|u\|_{\widetilde{L}^{\infty}_{t}(\dot{H}^{-2\delta})}+\|u\|_{\widetilde{L}^{1}_{t}(\dot{H}^{2(1-\delta)})}+\|\nabla\pi\|_{\widetilde{L}^{1}_{t}(\dot{H}^{-2\delta})}\\
		\leq& C\bigg\{\|u_{0}\|_{\dot{H}^{-2\delta}}+t^{\frac{1}{4}}2^{(\frac{3}{2}-2\delta)k}\|(a,b,\lambda)\|_{\widetilde{L}^{\infty}_{t}(B_{2,1}^{\frac{3}{2}})}\bigg(1+\|v\|_{L^{\infty}_{t}(\dot{B}_{3, 1}^{0})}\bigg)\\
		&\quad\times\|u\|_{L^{\infty}_{t}(\dot{B}_{3, 1}^{0})}^{\frac{1}{4}} \|u\|_{L^{1}_{t}(\dot{B}_{3, 1}^{2})}^{\frac{3}{4}}\bigg\}\exp\bigg\{\| v\|_{L^{1}_{t}(\dot{B}^{2}_{3,1})}+t2^{5k}\|(b,\lambda)\|^{2}_{\widetilde{L}^{\infty}_{t}(L^{2})}\bigg\}.
		\end{aligned}
		\end{align} 
	\end{Proposition}

	\begin{proof}
		We first get from  $1+a\geq \underline{b}$ and $1+b\geq\underline{b}$ that
		$$1+\dot{S}_{k}a=1+a+(\dot{S}_{k}a-a)\geq \frac{\underline{a}}{2}~{\rm and}~ 1+\dot{S}_{k}b=1+b+(\dot{S}_{k}b-b)\geq \frac{\underline{b}}{2}.$$
		Correspondingly, we rewrite the $u$ equation of (\ref{c54}) as 
		\begin{equation}\label{c57}
		\begin{aligned}
		\partial_{t}u+v\cdot\nabla u-\operatorname{div}(2(1+\dot{S}_{k}b)\mathbb{D}u)+\nabla\pi
		=E_{k}-\dot{S}_{k}a\nabla\pi-\nabla\dot{S}_{k}\lambda~\mathbb{D}u
		\end{aligned}
		\end{equation}
		with
		$$\begin{aligned}
		E_{k}=&\operatorname{div}(2(b-\dot{S}_{k}b)\mathbb{D}u)-\nabla(\lambda-\dot{S}_{k}\lambda)\mathbb{D}u-(a-\dot{S}_{k}a)\nabla\pi.
		\end{aligned}$$
		\par We now decompose the proof of Proposition \ref{Proposition-4.3} into the following steps:\\
		
		\noindent\textbf{Step 1.} The estimate of $\|u\|_{\widetilde{L}^{\infty}_{t}(\dot{H}^{-2\delta})}$ and $\|u\|_{L^{1}_{t}(\dot{H}^{2-2\delta})}$.\\
		
		Let $\mathbb{P}=I+\nabla(-\Delta)^{-1}\operatorname{div}$~be the Leray projection operator. Applying $\dot{\Delta}_{j}\mathbb{P}$ to $(\ref{c57})$ gives 
		\begin{equation}\label{c58}
		\begin{aligned}
		\partial_{t}\dot{\Delta}_{j}u+\dot{\Delta}_{j}\mathbb{P}(v\cdot\nabla u)-\dot{\Delta}_{j}\mathbb{P}\{\operatorname{div}(2(1+\dot{S}_{k}b)\mathbb{D}u)\}
		=\dot{\Delta}_{j}\mathbb{P}(E_{k}-\dot{S}_{k}a\nabla\pi-\nabla\dot{S}_{k}\lambda~\mathbb{D}u).
		\end{aligned}
		\end{equation}
		Multiplying the above equation by $\dot{\Delta}_{j}u$ and then integrating the resulting equations on $x\in\mathbb{R}^{3}$ leads to
		\begin{equation}\label{c59}
		\begin{aligned}
		\frac{d}{dt}\|\dot{\Delta}_{j}u\|_{L^{2}}+2^{2j}\|\dot{\Delta}_{j}u\|_{L^{2}}\lesssim&\|[v\cdot\nabla;\dot{\Delta}_{j}\mathbb{P}]u\|_{L^{2}}+2^{j}\|[\dot{S}_{k}b;\dot{\Delta}_{j}]\mathbb{D} u\|_{L^{2}} \\
		&+\|\dot{\Delta}_{j}E_{k}\|_{L^{2}}+\|\dot{\Delta}_{j}(\dot{S}_{k}a\nabla\pi)\|_{L^{2}}+\|\dot{\Delta}_{j}(\nabla\dot{S}_{k}\lambda~\mathbb{D}u)\|_{L^{2}}.
		\end{aligned}
		\end{equation}
		\par In what follows, we shall deal with the right-hand side of~(\ref{c59}).~Firstly applying homogeneous Bony's decomposition yields 
		$$[v\cdot\nabla;\dot{\Delta}_{j}\mathbb{P}]u=[T_{v}\cdot\nabla;\dot{\Delta}_{j}\mathbb{P}]u+T^{'}_{\nabla\dot{\Delta}_{j} u}v-\dot{\Delta}_{j}\mathbb{P}(T_{\nabla u}v)-\dot{\Delta}_{j}\mathbb{P}\operatorname{div}\mathcal{R}(u,v).$$
		It follows again from the above estimate,~which implies that
		$$
		\begin{aligned}
		\|[T_{v}\cdot\nabla;\dot{\Delta}_{j}\mathbb{P}]u\|_{L^{2}}
		&\lesssim\sum_{|j-k|\leq 4}2^{-j}\|\nabla\dot{S}_{k-1}v\|_{L^{\infty}}\|\dot{\Delta}_{k}\nabla u\|_{L^{2}}\\
		& \lesssim\sum_{|j-k|\leq 4} 2^{k-j} \|\nabla v\|_{L^{\infty}}\|\dot{\Delta}_{k}u\|_{L^{2}}\\
		& \lesssim 2^{2j\delta}\sum_{|j-k|\leq 4} 2^{(k-j)(1+2\delta)} \|\nabla v\|_{L^{\infty}}2^{-2k\delta}\|\dot{\Delta}_{k} u\|_{L^{2}}\\
		& \lesssim c_{j}2^{2j\delta}\|\nabla v\|_{L^{\infty}}\|u\|_{\dot {H}^{-2\delta}}.
		\end{aligned}
		$$
		The same estimate hold for $T^{'}_{\nabla\dot{\Delta}_{j} u}v.$ Note that
		$$
		\begin{aligned}
		\|\dot{\Delta}_{j}\mathbb{P}(T_{\nabla u} v)\|_{L^{2}}&\lesssim\sum_{|k-j|\leq 4}\|\dot{S}_{k-1}\nabla u\|_{L^{2}}\|\dot{\Delta}_{k}v\|_{L^{\infty}}\\&
		\lesssim 2^{2j\delta}\sum_{|k-j|\leq 4}2^{2(k-j)\delta}2^{-2k\delta}\|\dot{S}_{k-1} u\|_{L^{2}}\|\dot{\Delta}_{k}\nabla v\|_{L^{\infty}}\\&
		\lesssim c_{j}2^{2j\delta}\|u\|_{\dot {H}^{-2\delta}}\|\nabla v\|_{L^{\infty}}.
		\end{aligned}
		$$
		For $\dot{\Delta}_{j}\mathbb{P}\operatorname{div}\mathcal{R}(u,v),$ we have
		$$
		\begin{aligned}
		\|\dot{\Delta}_{j}\mathbb{P}\operatorname{div}\mathcal{R}(u,v)\|_{L^{2}}&\lesssim2^{2j}\|\dot{\Delta}_{j}\mathbb{P}\mathcal{R}(u,v)\|_{L^{\frac{6}{5}}}\\&
		\lesssim2^{2j}\sum_{k\geq j-3}\|\dot{\Delta}_{k} u\|_{L^{2}}\|\dot{\Delta}_{k} v\|_{L^{3}}\\&
		\lesssim 2^{2j\delta}\sum_{k\geq j-3}2^{(2-2\delta)(j-k)} 2^{-2k\delta}\|\dot{\Delta}_{k} u\|_{L^{2}}2^{k}\|\dot{\Delta}_{k}\nabla v\|_{L^{3}}\\&
		\lesssim c_{j}2^{2j\delta}\|u\|_{\dot {H}^{-2\delta}}\| \nabla v\|_{\dot{B}^{1}_{3,1}}.
		\end{aligned}
		$$
		From which, we obtain
		\begin{align}\label{c60}
		\|[v\cdot\nabla;\dot{\Delta}_{j}\mathbb{P}]u\|_{L^{2}}\lesssim c_{j}2^{2j\delta}\|\nabla v\|_{\dot{B}^{1}_{3,1}}\|u\|_{\dot {H}^{-2\delta}}.
		\end{align}
		The estimate of $[\dot{S}_{k}b;\dot{\Delta}_{j}]\mathbb{D}u,$ we get, by virtue of Lemma \ref{Lemma-2.2}, that
		\begin{align}\label{c61}
		\|[\dot{S}_{k}b;\dot{\Delta}_{j}]\mathbb{D}u\|_{L^{2}}\lesssim c_{j}2^{j(2\delta-1)}\|\nabla\dot{S}_{k}b\|_{L^{\infty}}\|u\|_{\dot{H}^{1-2\delta}}.
		\end{align}
		Similarly, for $\delta<\frac{3}{4}$, we obtain that
		\begin{align}\label{c62}
		\begin{aligned}
		\|\dot{\Delta}_{j} E_{k}\|_{L^{2}}\lesssim c_{j}2^{2j\delta} \|(a-\dot{S}_{k}a,~b-\dot{S}_{k}b,~\lambda-\dot{S}_{k}\lambda)\|_{\dot{B}^{\frac{3}{2}}_{2,1}}
		\bigg\{\|u\|_{\dot{H}^{2(1-\delta)}}+\|\nabla\pi\|_{\dot{H}^{-2\delta}}\bigg\}.
		\end{aligned}
		\end{align}
		Along the same line, one has
		\begin{equation}\label{c63}
		\begin{aligned}
		\|\dot{\Delta}_{j}(\dot{S}_{k}a\nabla\pi) \|_{L^{2}}
		\lesssim&\|\dot{\Delta}_{j}(T_{\nabla\pi}\dot{S}_{k}a) \|_{L^{2}}+\|\dot{\Delta}_{j}(T_{\dot{S}_{k}a}\nabla\pi) \|_{L^{2}}+\|\dot{\Delta}_{j}\mathcal{R}(\nabla\pi,\dot{S}_{k}a) \|_{L^{2}} \\
		\lesssim&c_j 2^{2j\delta}\|\dot{S}_{k}a\|_{\dot{B}^{-2\delta}_{\infty,2}}\|\nabla\pi\|_{L^{2}}+2^{\frac{3}{2}j}\|\dot{\Delta}_{j}\mathcal{R}(\nabla\pi,\dot{S}_{k}a) \|_{L^{1}}\\
		\lesssim&c_j 2^{2j\delta}\|\dot{S}_{k}a\|_{\dot{H}^{\frac{3}{2}-2\delta}}\|\nabla\pi\|_{L^{2}}.
		\end{aligned}
		\end{equation}
		Similarly,
		\begin{equation}\label{c663}
		\begin{aligned}
		\|\dot{\Delta}_{j}(\nabla\dot{S}_{k}\lambda~\mathbb{D}u)\|_{L^{2}}
		\lesssim&\|\dot{\Delta}_{j}(T_{\mathbb{D}u}\nabla\dot{S}_{k}\lambda) \|_{L^{2}}+\|\dot{\Delta}_{j}(T_{\nabla\dot{S}_{k}\lambda}\mathbb{D}u)\|_{L^{2}}+\|\dot{\Delta}_{j}\mathcal{R}(\mathbb{D}u,\nabla\dot{S}_{k}\lambda) \|_{L^{2}}\\
		\lesssim& c_j 2^{2j\delta}\|\nabla u\|_{\dot{H}^{-2\delta}}\|\nabla\dot{S}_{k}\lambda\|_{L^{\infty}}+2^{\frac{3}{2}j}\|\dot{\Delta}_{j}\mathcal{R}(\mathbb{D}u,\nabla\dot{S}_{k}\lambda) \|_{L^{1}}\\
		\lesssim&  c_j 2^{2j\delta}\|\nabla\dot{S}_{k}\lambda\|_{\dot{B}^{\frac{3}{2}}_{2,\infty}\cap L^{\infty}}\|u\|_{\dot{H}^{1-2\delta}}.
		\end{aligned}
		\end{equation}
		According to Lemma \ref{Lemma-2.1} and Lemma \ref{Lemma-2.2},~we may get
		\begin{align}\label{c664}
		\begin{aligned}
		&\|(\nabla\dot{S}_{k}b,~\nabla\dot{S}_{k}\lambda)\|_{\dot{B}^{\frac{3}{2}}_{2,\infty}\cap L^{\infty}}\lesssim 2^{\frac{5k}{2}}\|(\dot{S}_{k}b,~\dot{S}_{k}\lambda)\|_{L^{2}}.
		\end{aligned}
		\end{align}
		\par Substituting (\ref{c60})-(\ref{c663}) into~(\ref{c59}), using the interpolation inequality
		$\|u\|_{\dot{H}^{1-2\delta}}\lesssim\|u\|^{\frac{1}{2}}_{\dot{H}^{-2\delta}}$ $\|u\|^{\frac{1}{2}}_{\dot{H}^{2(1-\delta)}}$ and the Young's inequality, we write
		\begin{align}\label{c64}
		\begin{aligned}
		&\|u\|_{\widetilde{L}^{\infty}_{t}(\dot{H}^{-2\delta})}+\|u\|_{\widetilde{L}^{1}_{t}(\dot{H}^{2(1-\delta)})}\\
		\lesssim& \|u_{0}\|_{\dot{H}^{-2\delta}}+\int^{t}_{0}\bigg(\|\nabla v\|_{\dot{B}^{1}_{3,1}}+2^{5k}\|(\dot{S}_{k}b,\dot{S}_{k}\lambda)\|^{2}_{L^{2}}\bigg)\|u\|_{\widetilde{L}^{\infty}_{t}(\dot{H}^{-2\delta})}d\tau\\
		&+\|(a-\dot{S}_{k}a,b-\dot{S}_{k}b,\lambda-\dot{S}_k\lambda)\|_{\widetilde{L}^{\infty}_{T}(\dot{B}^{\frac{3}{2}}_{2,1})}\bigg\{\|u\|_{\widetilde{L}^{1}_{t}(\dot{H}^{2(1-\delta)})}+\|\nabla\pi\|_{\widetilde{L}^{1}_{t}(\dot{H}^{-2\delta})}\bigg\}\\&+2^{(\frac{3}{2}-2\delta)k}\|\dot{S}_{k}a\|_{L^{\infty}_{t}(L^{2})}\|\nabla\pi\|_{L^{1}_{t}(L^{2})}.
		\end{aligned}
		\end{align}
		
		\noindent\textbf{Step 2.} The estimate of $\|\nabla\pi\|_{\widetilde{L}^{1}_{t}(\dot{H}^{-2\delta})}.$\\
		
		In order to estimate the pressure function $\pi,$ we get from $(\ref{c54})_{1}$ that
		\begin{equation}\label{c65}
		\begin{aligned}
		\operatorname{div}((1+\dot{S}_{k}a)\nabla \pi)=&-\operatorname{div}(v\cdot\nabla u)+\operatorname{div}\operatorname{div}(2\dot{S}_{k}b\mathbb{D}u)+\operatorname{div}E_{k}-\operatorname{div}(\nabla\dot{S}_{k}\lambda~\mathbb{D}u),
		\end{aligned}
		\end{equation}
		which implies
		$$
		\begin{aligned}
		\operatorname{div}((1+\dot{S}_k a) \nabla \dot{\Delta}_j \pi)=&-\operatorname{div} \dot{\Delta}_j(v \cdot \nabla u)+\operatorname{div} \dot{\Delta}_j E_{k}\\
		&+\operatorname{div} \dot{\Delta}_j(2\nabla \dot{S}_k b ~\mathbb{D}u+\dot{S}_k b \Delta u)-\operatorname{div}\dot{\Delta}_j (\nabla\dot{S}_{k}\lambda~\mathbb{D}u)\\&-\operatorname{div}[\dot{\Delta}_j, \dot{S}_k a] \nabla \pi.
		\end{aligned}
		$$
		Taking $L^{2}$ inner product of the above equation with $\dot{\Delta}_j \pi$ and using a similar argument as (\ref{c59}), we find
		\begin{align}\label{c66}
		\begin{aligned}
		\|\nabla\pi\|_{\widetilde{L}^{1}_{t}(\dot{H}^{-2\delta})}\lesssim&	\|\operatorname{div}(v\cdot\nabla u)\|_{\widetilde{L}^{1}_{t}(\dot{H}^{-1-2\delta})}+\|E_{k}\|_{\widetilde{L}^{1}_{t}(\dot{H}^{-2\delta})}\\&+\|\nabla \dot{S}_k b ~\mathbb{D}u\|_{\widetilde{L}^{1}_{t}(\dot{H}^{-2\delta})}+\|\Delta u\cdot\nabla\dot{S}_k b\|_{\widetilde{L}^{1}_{t}(\dot{H}^{-1-2\delta})}\\&+\|\nabla\dot{S}_{k}\lambda~\mathbb{D}u\|_{\widetilde{L}^{1}_{t}(\dot{H}^{-2\delta})}+\|(2^{-2j\delta }\|[\dot{S}_k a ;\dot{\Delta}_j] \nabla \pi \|_{L^1_{t}(L^2)})_{j\in\mathbb{Z}}\|_{l^2}.
		\end{aligned}
		\end{align}
		We now estimate term by term in (\ref{c66}). Then applying Lemma \ref{Lemma-2.2} yields
		\begin{align}\label{cc20}
		\|\nabla v\cdot\nabla u\|_{\widetilde{L}^{1}_{t}(\dot{H}^{-1-2\delta})}\lesssim\int^{t}_{0}\|\nabla v\|_{\dot{B}^{1}_{3,1}}\|u\|_{\widetilde{L}^{\infty}_{t}(\dot{H}^{-2\delta})}d\tau,
		\end{align}
		and
		\begin{align}
		\begin{aligned}
		\|E_{k}\|_{\widetilde{L}^{1}_{t}(\dot{H}^{-2\delta})}\lesssim& \|(a-\dot{S}_{k}a,~b-\dot{S}_{k}b,~\lambda-\dot{S}_{k}\lambda)\|_{\widetilde{L}^{\infty}_{T}(\dot{B}^{\frac{3}{2}}_{2,1})}\bigg\{\|u\|_{\widetilde{L}^{1}_{t}(\dot{H}^{2(1-\delta)})}+\|\nabla\pi\|_{\widetilde{L}^{1}_{t}(\dot{H}^{-2\delta})}\bigg\}.
		\end{aligned}
		\end{align}
		Similarly,
		\begin{align}
		\begin{aligned}
		&\|\nabla \dot{S}_k b~\mathbb{D}u\|_{\widetilde{L}^{1}_{t}(\dot{H}^{-2\delta})}
		\lesssim \|\nabla\dot{S}_{k}b\|_{L^{2}_{t}(\dot{B}^{\frac{3}{2}}_{2,\infty}\cap L^{\infty})}\|u\|_{L^{2}_{t}(\dot {H}^{1-2\delta})},
		\end{aligned}
		\end{align}
		and
		\begin{align}\label{c67}
		\|\Delta u\cdot\nabla\dot{S}_k b\|_{\widetilde{L}^{1}_{t}(\dot{H}^{-1-2\delta})}\lesssim \|\nabla\dot{S}_{k}b\|_{L^{2}_{t}(\dot{B}^{\frac{3}{2}}_{2,\infty}\cap L^{\infty})}\|u\|_{L^{2}_{t}(\dot {H}^{1-2\delta})}.
		\end{align}
		The same estimate holds for $\nabla\dot{S}_{k}\lambda~\mathbb{D}u,$ one has
		\begin{align}
		\|\nabla\dot{S}_{k}\lambda~\mathbb{D}u\|_{\widetilde{L}^{1}_{t}(\dot{H}^{-2\delta})}\lesssim \|\nabla\dot{S}_{k}\lambda\|_{L^{2}_{t}(\dot{B}^{\frac{3}{2}}_{2,\infty}\cap L^{\infty})}\|u\|_{L^{2}_{t}(\dot {H}^{1-2\delta})}.
		\end{align}
		Whereas applying Bony’s decomposition once again leads to
		$$[\dot{S}_k a ;\dot{\Delta}_j] \nabla \pi=[T_{\dot{S}_k a};\dot{\Delta}_j]\nabla \pi+T^{'}_{\dot{\Delta}_j \nabla \pi}\dot{S}_k a-\dot{\Delta}_j T_{\nabla \pi} \dot{S}_k a-\dot{\Delta}_j\mathcal{R}( \dot{S}_k a,\nabla \pi).$$
		Applying Lemma \ref{Lemma-2.2} yields
		$$
		\begin{aligned}
		\|[T_{\dot{S}_k a};\dot{\Delta}_j]\nabla \pi\|_{L^{2}}
		&\lesssim\sum_{|j-l|\leq 4}2^{-j}\|\nabla\dot{S}_{l-1}\dot{S}_{k}a\|_{L^{\infty}}\|\dot{\Delta}_{l}\nabla \pi\|_{L^{2}}\\
		&\lesssim c_{j}2^{2j\delta}\|\dot{S}_{k}a\|_{\dot{H}^{\frac{3}{2}-2\delta}}\|\nabla \pi\|_{L^{2}}.
		\end{aligned}
		$$
		Similarly, we have 
		$$
		\begin{aligned}
		\|T^{'}_{\dot{\Delta}_j \nabla \pi}\dot{S}_k a\|_{L^{2}}&\lesssim\sum_{l\geq j-2}\|\dot{S}_{l+2}\dot{\Delta}_{j}\nabla \pi\|_{L^{\infty}}\|\dot{\Delta}_{l}\dot{S}_{k}a\|_{L^{2}}\\
		&\lesssim2^{2j\delta}\|\dot{\Delta}_{j}\nabla \pi\|_{L^{2}} \sum_{l\geq j-2}2^{(j-l)(\frac{3}{2}-2\delta)}2^{l(\frac{3}{2}-2\delta)}\|\dot{\Delta}_{l}\dot{S}_{k}a\|_{L^{2}}\\
		&\lesssim c_{j}2^{2j\delta}\|\nabla \pi\|_{L^{2}}\|\dot{S}_{k}a\|_{\dot{H}^{\frac{3}{2}-2\delta}}.
		\end{aligned}
		$$
		The same estimate holds for $\dot{\Delta}_j T_{\nabla \pi} \dot{S}_k a$. Notice that
		$$
		\begin{aligned}
		\|\dot{\Delta}_j\mathcal{R}( \dot{S}_k a,\nabla \pi)\|_{L^{2}}&\lesssim2^{\frac{3}{2}j}\|\dot{\Delta}_j\mathcal{R}( \dot{S}_k a,\nabla \pi)\|_{L^{1}}\\
		&\lesssim2^{2j\delta}\sum_{l\geq j-3}2^{(j-l)(\frac{3}{2}-2\delta)}\|\dot{\Delta}_{l}\nabla \pi\|_{L^{2}}2^{l(\frac{3}{2}-2\delta)}\|\dot{\Delta}_{l}\dot{S}_{k}a\|_{L^{2}}\\
		&\lesssim c_{j}2^{2j\delta}\|\nabla \pi\|_{L^{2}}\|\dot{S}_{k}a\|_{\dot{H}^{\frac{3}{2}-2\delta}}.
		\end{aligned}
		$$
		Thus, we deduce that
		\begin{align}\label{c68}
		\|[\dot{S}_k a ;\dot{\Delta}_j] \nabla \pi\|_{L^{1}_{t}(L^{2})}\lesssim c_{j}2^{2j\delta}\|\dot{S}_{k}a\|_{\widetilde{L}^{\infty}_{t}(\dot{H}^{\frac{3}{2}-2\delta})}\|\nabla \pi\|_{L^{1}_{t}(L^{2})}.
		\end{align}
		\par Substituting (\ref{cc20})-(\ref{c68}) into (\ref{c66}) and using the interpolation inequality
		$\|u\|_{\dot{H}^{1-2\delta}}\lesssim\|u\|^{\frac{1}{2}}_{\dot{H}^{-2\delta}}\|u\|^{\frac{1}{2}}_{\dot{H}^{2(1-\delta)}}$ and Young's inequality, we write
		\begin{align}\label{c69}
		\begin{aligned}
		\|\nabla\pi\|_{\widetilde{L}^{1}_{t}(\dot{H}^{-2\delta})}\leq&C\int^{t}_{0}\bigg(\|\nabla v\|_{\dot{B}^{1}_{3,1}}+C2^{5k}\|(\dot{S}_{k}b,\dot{S}_{k}\lambda)\|^{2}_{L^{2}}\bigg)\|u\|_{\widetilde{L}^{\infty}_{t}(\dot{H}^{-2\delta})}d\tau\\
		&+C\|(a-\dot{S}_{k}a,b-\dot{S}_{k}b,\lambda-\dot{S}_k \lambda)\|_{\widetilde{L}^{\infty}_{T}(\dot{B}^{\frac{3}{2}}_{2,1})}\bigg\{\|u\|_{\widetilde{L}^{1}_{t}(\dot{H}^{2(1-\delta)})}+\|\nabla\pi\|_{\widetilde{L}^{1}_{t}(\dot{H}^{-2\delta})}\bigg\}\\
		&+C\|\dot{S}_{k}a\|_{\widetilde{L}^{\infty}_{t}(\dot{H}^{\frac{3}{2}-2\delta})}\|\nabla \pi\|_{L^{1}_{t}(L^{2})}+\frac{1}{8}\|u\|_{\widetilde{L}^{1}_{t}(\dot{H}^{2-2\delta})},
		\end{aligned}
		\end{align}
		which along with (\ref{c64}) ensures that
		$$
		\begin{aligned}
		&\|u\|_{\widetilde{L}^{\infty}_{t}(\dot{H}^{-2\delta})}+\|u\|_{\widetilde{L}^{1}_{t}(\dot{H}^{2(1-\delta)})}+\|\nabla\pi\|_{\widetilde{L}^{1}_{t}(\dot{H}^{-2\delta})}\\
		\leq&\|u_{0}\|_{\dot{H}^{-2\delta}}+C\int^{t}_{0}\bigg(\|\nabla v\|_{\dot{B}^{1}_{3,1}}+2^{5k}\|(\dot{S}_{k}b,\dot{S}_{k}\lambda)\|^{2}_{L^{2}}\bigg)\|u\|_{\widetilde{L}^{\infty}_{t}(\dot{H}^{-2\delta})}d\tau\\
		&+C\|(a-\dot{S}_{k}a,b-\dot{S}_{k}b,\lambda-\dot{S}_k \lambda)\|_{\widetilde{L}^{\infty}_{T}(\dot{B}^{\frac{3}{2}}_{2,1})}\bigg\{\|u\|_{\widetilde{L}^{1}_{t}(\dot{H}^{2(1-\delta)})}+\|\nabla\pi\|_{\widetilde{L}^{1}_{t}(\dot{H}^{-2\delta})}\bigg\}\\
		&+C\|\dot{S}_{k}a\|_{\widetilde{L}^{\infty}_{t}(\dot{H}^{\frac{3}{2}-2\delta})}\|\nabla \pi\|_{L^{1}_{t}(L^{2})}.
		\end{aligned}
		$$
		Then we have used $a,$ $b,$ $\lambda\in \widetilde{L}^{\infty}_{T}(B^{\frac{3}{2}}_{2,1})$ and chosen $k$ sufficiently large so that
		\begin{equation}
		C\|(a-\dot{S}_{k}a,b-\dot{S}_{k}b,\lambda-\dot{S}_k \lambda)\|_{\widetilde{L}^{\infty}_{t}(\dot{B}^{\frac{3}{2}}_{2,1})}\leq\frac{1}{8},
		\end{equation}
		from which, we duduce that
		\begin{align}\label{c70}
		\begin{aligned}
		&\|u\|_{\widetilde{L}^{\infty}_{t}(\dot{H}^{-2\delta})}+\|u\|_{\widetilde{L}^{1}_{t}(\dot{H}^{2(1-\delta)})}+\|\nabla\pi\|_{\widetilde{L}^{1}_{t}(\dot{H}^{-2\delta})}\\
		\leq&\|u_{0}\|_{\dot{H}^{-2\delta}}+C\int^{t}_{0}\bigg(\|\nabla v\|_{\dot{B}^{1}_{3,1}}+2^{5k}\|(\dot{S}_{k}b,\dot{S}_{k}\lambda)\|^{2}_{L^{2}}\bigg)\|u\|_{\widetilde{L}^{\infty}_{t}(\dot{H}^{-2\delta})}d\tau\\
		&+C\|\dot{S}_{k}a\|_{\widetilde{L}^{\infty}_{t}(\dot{H}^{\frac{3}{2}-2\delta})}\|\nabla \pi\|_{L^{1}_{t}(L^{2})}.
		\end{aligned}
		\end{align}
		
		\noindent\textbf{Step 3.} The estimate of $\|\nabla\pi\|_{L^{1}_{t}(L^{2})}.$\\
		
		\par We first get by taking div to $(\ref{c54})_{1}$ that
		$$
		\begin{aligned}
		\operatorname{div}\bigg\{\left(1+a\right)\nabla\pi\bigg\}=&-\operatorname{div}\left(v\cdot\nabla u\right)+\operatorname{div}\operatorname{div} \left(2b\mathbb{D}u\right)
		-\operatorname{div}\left(\nabla\lambda~\mathbb{D}u\right),
		\end{aligned}
		$$
		from which and $0<\underline{a}\leq 1+a$, we deduce by a standard energy estimate that
		\begin{equation}\label{b25}
		\begin{aligned}
		\|\nabla\pi\|_{L^{1}_{t}(L^{2})}\lesssim&\|v\cdot\nabla u\|_{L^{1}_{t}(L^{2})}+\|\Delta u\cdot\nabla b\|_{L^{1}_{t}(\dot{H}^{-1})}\\&+\|\operatorname{div}(\nabla b~\mathbb{D}u)\|_{L^{1}_{t}(\dot{H}^{-1})}+\|\operatorname{div}\left(\nabla\lambda~\mathbb{D}u\right)\|_{L^{1}_{t}(\dot{H}^{-1})}.
		\end{aligned}
		\end{equation}
		Then it follows from interpolation inequalities in Besov spaces that
		\begin{equation}\label{b26}
		\begin{aligned}
		\|v\cdot\nabla u\|_{L^{2}}\lesssim &\|v\|_{L^{3}}\|\nabla u\|_{L^{6}}\lesssim \|v\|_{\dot{B}_{3, 1}^{0}}\|\nabla u\|_{\dot{B}_{3, 1}^{\frac{1}{2}}}
		\lesssim\|v\|_{\dot{B}_{3, 1}^{0}}\|u\|_{\dot{B}_{3, 1}^{0}}^{\frac{1}{4}}\|u\|_{\dot{B}_{3, 1}^{2}}^{\frac{3}{4}}.
		\end{aligned}
		\end{equation}
		Similarly, we can deduce 
		\begin{equation}\label{b27}
		\begin{aligned}
		\|\Delta u\cdot\nabla b\|_{\dot{H}^{-1}}\lesssim&\|T_{\nabla b}\Delta u\|_{\dot{H}^{-1}}+\|T_{\Delta u}\nabla b\|_{\dot{H}^{-1}}+\|\operatorname{div}\mathcal{R}(b,\Delta u)\|_{\dot{H}^{-1}}\\
		\lesssim&\|\nabla b\|_{L^{3}}\|\Delta u\|_{\dot{B}_{6, 2}^{-1}}+\|\mathcal{R}(b,\Delta u)\|_{\dot{B}^{\frac{1}{2}}_{\frac{3}{2},2}} \\
		\lesssim&\|b\|_{\dot{B}_{3, 1}^{1}}\|u\|_{\dot{B}_{3, 1}^{0}}^{\frac{1}{4}}\|u\|_{\dot{B}_{3, 1}^{2}}^{\frac{3}{4}},
		\end{aligned}
		\end{equation}	
		and
		\begin{equation}\label{b28}
		\begin{aligned}
		\|\operatorname{div}(\nabla b~\mathbb{D}u)\|_{\dot{H}^{-1}}\lesssim\|\nabla u\|_{L^{6}}\|\nabla b\|_{L^{3}} \lesssim\| b\|_{\dot{B}_{3, 1}^{1}}\|u\|_{\dot{B}_{3, 1}^{0}}^{\frac{1}{4}}\|u\|_{\dot{B}_{3, 1}^{2}}^{\frac{3}{4}}.
		\end{aligned}
		\end{equation}	
		Along the same way, we obtain
		\begin{equation}\label{b29}
		\begin{aligned}
		\|\operatorname{div}\left(\nabla\lambda~\mathbb{D}u\right)\|_{\dot{H}^{-1}}
		\lesssim \|\lambda\|_{\dot{B}_{3, 1}^{1}}\|u\|_{\dot{B}_{3, 1}^{0}}^{\frac{1}{4}}\|u\|_{\dot{B}_{3, 1}^{2}}^{\frac{3}{4}}.
		\end{aligned}
		\end{equation}
		\par Substituting (\ref{b26}), (\ref{b27}), (\ref{b28}) and (\ref{b29}) into (\ref{b25}) and the fact that
		\begin{equation}\label{b30}
		\begin{aligned}
		\|\nabla\pi\|_{L^{1}_{t}(L^{2})}\lesssim&\int^{t}_{0}\|v\|_{\dot{B}_{3, 1}^{0}}\|u\|_{\dot{B}_{3, 1}^{0}}^{\frac{1}{4}} \|u\|_{\dot{B}_{3, 1}^{2}}^{\frac{3}{4}}d\tau+\|(b,\lambda)\|_{L^{\infty}_{t}(\dot{B}_{3, 1}^{1})}\int^{t}_{0}\|u\|_{\dot{B}_{3, 1}^{0}}^{\frac{1}{4}}\|u\|_{\dot{B}_{3, 1}^{2}}^{\frac{3}{4}} d\tau,
		\end{aligned}
		\end{equation}
		which combining (\ref{c55}) and (\ref{c70}) and using the Young's inequality, for some integer $k$, we achieve
		\begin{align}\label{ddd1}
		\begin{aligned}
		&\|u\|_{\widetilde{L}^{\infty}_{t}(\dot{H}^{-2\delta})}+\|u\|_{\widetilde{L}^{1}_{t}(\dot{H}^{2(1-\delta)})}+\|\nabla\pi\|_{\widetilde{L}^{1}_{t}(\dot{H}^{-2\delta})}\\
		\lesssim& \|u_{0}\|_{\dot{H}^{-2\delta}}+\int^{t}_{0}\bigg(\|\nabla v\|_{\dot{B}^{1}_{3,1}}+C2^{5k}\|(b,\lambda)\|^{2}_{L^{2}}\bigg)\|u\|_{\widetilde{L}^{\infty}_{t}(\dot{H}^{-2\delta})}d\tau\\&+2^{(\frac{3}{2}-2\delta)k}\|a\|_{L^{\infty}_{t}(L^{2})}\int^{t}_{0}\|v\|_{\dot{B}_{3, 1}^{0}}\|u\|_{\dot{B}_{3, 1}^{0}}^{\frac{1}{4}} \|u\|_{\dot{B}_{3, 1}^{2}}^{\frac{3}{4}}d\tau\\&+2^{(\frac{3}{2}-2\delta)k}\|a\|_{L^{\infty}_{t}(L^{2})}\|(b,\lambda)\|_{L^{\infty}_{t}(\dot{B}_{3, 1}^{1})}\int^{t}_{0}\|u\|_{\dot{B}_{3, 1}^{0}}^{\frac{1}{4}}\|u\|_{\dot{B}_{3, 1}^{2}}^{\frac{3}{4}} d\tau.
		\end{aligned}
		\end{align}
		Finally, applying Gronwall's Lemma to \eqref{ddd1}, it implies the desired inequality (\ref{c56}).
	\end{proof}
	
	\par  With Lemma \ref{Lemma-2}, we can prove the propagation of regularities of $u_0\in \dot{H}^{-2\delta}$ for $u$ on $[0,1]$, even higher regularity estimates can be obtained with the following results:

	\begin{Corollary}\label{Corollary-3.1}
		Under the assumptions of Theorem \ref{Theorem-1}, there exists $t_0\in]0,1[$ such that $u(t_0)\in \dot{H}^{-2\delta}\cap\dot{B}^{\frac{5}{2}}_{2,1}(\mathbb{R}^3).$ Moreover, there holds 
		\begin{equation}\label{c72}
		\|u(t_{0})\|_{\dot{H}^{-2\delta}\cap \dot{B}^{\frac{5}{2}}_{2,1}}\leq C(1+\|u_{0}\|_{\dot{H}^{-2\delta}}),
		\end{equation}
		where $C$ depends only on $\underline{\mu}$, $\overline{\mu}$, $\|\rho_0-1\|_{{B}_{2,1}^{\frac{3}{2}}}$ and $\|u_0\|_{\dot{B}^{\frac{1}{2}}_{2,1}}.$
	\end{Corollary}

	\begin{proof} 
		Thanks to Lemma \ref{Lemma-2}, we can find some $t_{0}\in (0,1)$ such that $u(t_{0})\in\dot{B}^{\frac{1}{2}}_{2,1}\cap\dot{B}^{\frac{5}{2}}_{2,1}$ and 
		$$
		\|u(t_{0})\|_{\dot{B}^{\frac{1}{2}}_{2,1}\cap\dot{B}^{\frac{5}{2}}_{2,1}}\leq C\|u_{0}\|_{\dot{B}^{\frac{1}{2}}_{2,1}}.
		$$
		To obtain that (\ref{c72}) holds, we need to resort to the result of the Lemma \ref{Lemma-2} and check that it satisfies the conditions of Proposition \ref{Proposition-4.3}, i.e., $b,$ $\lambda\in \widetilde{L}^{\infty}_{T}(B^{\frac{3}{2}}_{2,1}).$ By virtue of the definition of $b(a)$ and $\lambda(a)$ from (\ref{a3}), we have $b(0)=\lambda(0)=0.$ Thus, resorting to the Paralinearization theorem of \cite{2011BCD}, we can easily obtain
		\begin{align}
		\|(a,b,\lambda)\|_{\widetilde{L}^{\infty}_{1}(B^{\frac{3}{2}}_{2,1})}\lesssim\|a\|_{\widetilde{L}^{\infty}_{1}(B^{\frac{3}{2}}_{2,1})}\lesssim\|a_0\|_{B^{\frac{3}{2}}_{2,1}}\lesssim \|\rho_0-1\|_{B^{\frac{3}{2}}_{2,1}}.
		\end{align}
		Therefore, there exists an integer $k\in\mathbb{Z},$ such that, for any small positive constant $c$, such that
		\begin{align}
		\|(a-\dot{S}_{k}a,~b-\dot{S}_{k}b,~\lambda-\dot{S}_{k}\lambda)\|_{\widetilde{L}^{\infty}_{T}(B^{\frac{3}{2}}_{2,1})}\leq c.
		\end{align}
		from which we can conclude that combining Proposition \ref{Proposition-4.3} gives (\ref{c72}).

	\end{proof}
	\section{Global well-posedness of (\ref{a1})}
	The goal of this section is to prove the global well-posedness part of Theorem \ref{Theorem-1} provided that $\|u_{0}\|_{\dot{B}^{\frac{1}{2}}_{2,1}}$ is sufficiently small. As a convention in the remaining of this section, we shall always  denote $t_0$ to be the positive time determined by Corollary \ref{Corollary-3.1}. \\
	
	\noindent\textbf{Strategy of the proof of Theorem \ref{Theorem-1}:} Thanks to Lemma \ref{Lemma-2}, we conclude that: given $\rho_0-1\in B^{\frac{3}{2}}_{2,1}(\mathbb{R}^{3})$ and $u_{0}\in \dot{B}^{\frac{1}{2}}_{2,1}(\mathbb{R}^{3})$ with $\|u_{0}\|_{\dot{B}^{\frac{1}{2}}_{2,1}}$ sufficiently small, (\ref{a1}) has a unique local solution $(\rho, u)$ satisfying $\rho-1\in\mathcal{C}([0,T^{*});B^{\frac{3}{2}}_{2,1}(\mathbb{R}^{3}))$ and $u\in \mathcal{C}([0,T^{*});\dot{B}^{\frac{1}{2}}_{2,1}(\mathbb{R}^{3}))\cap L^{1}_{loc}([0,T^{*});\dot{B}^{\frac{5}{2}}_{2,1}(\mathbb{R}^{3}))$ for some $T^{*}\geq 1$, and we can find some $t_{0}\in (0,1)$ such that
	\begin{equation}\label{d1}
	\|u(t_{0})\|_{\dot{B}^{\frac{1}{2}}_{2,1}\cap\dot{B}^{\frac{5}{2}}_{2,1}}\leq C\|u_{0}\|_{\dot{B}^{\frac{1}{2}}_{2,1}}.
	\end{equation}
	Notice from (\ref{d1}) that $\|u(t_{0})\|_{\dot{B}^{\frac{1}{2}}_{2,1}\cap\dot{B}^{\frac{5}{2}}_{2,1}}$ is very small provided that $\|u_{0}\|_{\dot{B}^{\frac{1}{2}}_{2,1}}$ is sufficiently small. Moreover, applying Corollary \ref{Corollary-3.1} gives rise to
	\begin{equation}\label{d2}
	\begin{aligned}
	\|u(t_0)\|_{\dot{H}^{-2\delta}\cap \dot{B}^{\frac{5}{2}}_{2,1}}\lesssim1+\|u_{0}\|_{\dot{H}^{-2\delta}}.
	\end{aligned}
	\end{equation}
	\par Our aim of what follows is to prove that $T^{*}=\infty,$ we rewrite the $u$ equation as
	$$
	\Delta u-\nabla(\frac{\pi}{\mu(\rho)})=-\frac{2\mathbb{D} u\cdot\nabla\mu(\rho)}{\mu(\rho)}+\frac{\rho\partial_t u+ \rho u\cdot\nabla u}{\mu(\rho)}+\frac{\pi\nabla\mu(\rho)}{\mu(\rho)^{2}},
	$$
	the key ingredient here here is to find the time-independent bounds on the $L^{\infty}(0,\infty;L^{q})$ norm of $\nabla\mu(\rho)$ using the result of Lemma \ref{Lemma-5.1}. So it is necessary to present the \emph{a priori} $L^{1}([t_0,\infty);L^\infty(\mathbb{R}^{3}))$ estimate for $\nabla u$ is sufficiently small, which is a crucial part of the proof of Theorem \ref{Theorem-1}. The new ingredient here is to decompose the velocity fields $u$ into $v$ and $w$, where $v$ satisfies the 3D classical Navier-Stokes equations:
	\begin{equation}\label{dd1}
	\left\{\begin{array}{l}
	\partial_t v+v \cdot \nabla v-\Delta v+\nabla \pi_v=0, \\
	\operatorname{div} v=0, \\
	\left. v\right|_{t=t_0}=u\left(t_0\right),
	\end{array}\right.
	\end{equation}
	and then solve $w=u-v$ via
	\begin{equation}\label{dd2}
	\left\{\begin{array}{l}
	\partial_t \rho+\operatorname{div}(\rho(v+w))=0, \\
	\rho \partial_t w+\rho(v+w) \cdot \nabla w-\operatorname{div}(2\mu(\rho)\mathbb{D}w)+\nabla \pi_w\\
	=(1-\rho)\left(\partial_t v+v \cdot \nabla v\right)-\rho w \cdot \nabla v+\operatorname{div}(2(\mu(\rho)-1)\mathbb{D}v), \\
	\operatorname{div} w=0, \\
	\left.\rho\right|_{t=t_0}=\rho\left(t_0\right),\left.\quad w\right|_{t=t_0}=0,
	\end{array}\right.
	\end{equation}
	which can be reached through energy estimate in the $L^2$ framework. The detailed information of $v$ is presented in
	Proposition \ref{Proposition-5.1}, and that of $w$ is in Lemma \ref{Lemma-5.2} and Lemma \ref{Lemma-5.3}.
	\vskip 0.5cm
	\par In order to get the global solution of (\ref{dd1}), we need to recall the following Proposition \cite{2013AGZ}, where we omitted the proof.
	\begin{Proposition}\label{Proposition-5.1}(see Proposition 5.1 of \cite{2013AGZ}.) 
		Let $(v,\pi_{v})$ be a unqiue global solution of $(\ref{dd1})$ which satisfies $(\ref{a7})$. Then for $s\in[\frac{1}{2},\frac{5}{2}],$ there hold
		\begin{equation}\label{d3}
		\|v\|_{\widetilde{L}^{\infty}([t_0,\infty);\dot{B}^{s}_{2,1})}+\| v\|_{L^1([t_0,\infty);\dot{B}^{s+2}_{2,1})}\leq C\|u(t_0)\|_{\dot{B}_{2,1}^{s}}\leq C\|u_0\|_{\dot{B}^{\frac{1}{2}}_{2,1}}.
		\end{equation}
		and
		\begin{equation}\label{d4}
		\|\partial_t v\|_{L^\infty([t_0,\infty);\dot{B}^{\frac{1}{2}}_{2,1})}+\|\partial_t v\|_{L^{1}([t_0,\infty);\dot{B}^{\frac{5}{2}}_{2,1})}\leq C\|u(t_0)\|_{\dot{B}_{2,1}^{s}}\leq C\|u_0\|_{\dot{B}^{\frac{1}{2}}_{2,1}}.
		\end{equation} 
	\end{Proposition}
	\vskip 0.5cm
	The goal of this section is to prove that the following Proposition holds, which is the most important ingredient used in the proof of Theorem \ref{Theorem-1}, i.e.,
	
	\begin{Proposition}\label{Proposition-5.5}
		Let $M\stackrel{\mathrm{ def }}{=}\|\nabla\mu(\rho_0)\|_{L^{q}}$ with $q\in]3,6[$, $n\in]2,\frac{6\delta_{-}}{1+\delta_{-}}[$ with $\delta\in]1/2,3/4[$ and $m\in(3,\min\{r,6\}).$ Then there exists some positive constant $\varepsilon,$ which depends only on $n,m,\bar{\mu}, \underline{\mu}$ and $M$ and $\|u_0\|_{\dot{H}^{-2\delta}}$ such that if $u$ and $w$ are the unique local strong solution of $(\ref{a1})$ and $(\ref{dd2})$ on $\mathbb{R}^{3}\times[t_0,T]$ respectively, and satisfying
		\begin{equation}\label{d39}
		\sup_{t\in[t_0,T]}\|\nabla\mu(\rho)\|_{L^{q}}\leq 4M\quad{\rm and}\quad
		\sup_{t\in[t_{0},T]}\|\nabla w\|^{2}_{L^{2}}\leq  4C\|u_{0}\|^{2}_{\dot{B}^{\frac{1}{2}}_{2,1}},
		\end{equation}
		then the following estimates hold:
		\begin{equation}\label{d40}
		\sup_{t\in[t_0,T]}\|\nabla\mu(\rho)\|_{L^{q}}\leq 2M\quad{\rm and}\quad
		\sup_{t\in[t_{0},T]}\|\nabla w\|^{2}_{L^{2}}\leq 2C\|u_{0}\|^{2}_{\dot{B}^{\frac{1}{2}}_{2,1}},
		\end{equation}
		provided that $\|u_0\|_{\dot{B}^{\frac{1}{2}}_{2,1}}\leq \varepsilon.$
	\end{Proposition}
	
	\par Before proving Proposition \ref{Proposition-5.5}, we establish some necessary \emph{a priori} estimates, see Lemmas {\ref{Lemma-u1}-\ref{Lemma-5.6}}. 
	\vskip 0.5cm
	\begin{Corollary}\label{Corollary-5.1}
		Assuming that $	\sup_{t\in[0,T]}\|\nabla\mu(\rho)\|_{L^{q}}\leq 4\|\nabla\mu(\rho_0)\|_{L^{q}}$ with $q\in(3,6)$ is valid, then there holds that
		\begin{equation}\label{dd36}
		\|\nabla^{2} u\|_{L^{2}}\leq C \|\rho u_t\|_{L^{2}}+C\|\nabla u\|_{L^{2}}+C\|\nabla u\|^{3}_{L^{2}},
		\end{equation}
		and
		\begin{equation}\label{d36}
		\|\nabla^{2} w\|_{L^{2}}\leq C\|\rho w_t\|_{L^{2}}+C\|\nabla w\|_{L^{2}}+C\|\nabla w\|^{3}_{L^{2}}+C\|v\|_{\dot{B}^{s}_{2,1}},
		\end{equation}
		where $C$ depends only on $\underline{\mu}$, $\overline{\mu}$, $q$ and $\|\nabla\mu(\rho_0)\|_{L^{q}}$. 
	\end{Corollary}
	\begin{proof} 
		The momentum equations $(\ref{a1})_2$ can be rewritten as follows,
		\begin{equation}
		-\operatorname{div}(2\mu(\rho)\mathbb{D}u)+\nabla\pi=-\rho u_t-\rho u\cdot\nabla u.
		\end{equation}
		Thanks to Lemma \ref{Lemma-5.1}, we achieve
		\begin{equation}
		\begin{aligned}
		\|\nabla^{2} u\|_{L^{2}}
		&\lesssim \|\rho u_t\|_{L^{2}}+\|\nabla u\|_{L^{2}}+\|u\|_{L^6}\|\nabla u\|_{L^2}\\
		&\lesssim \|\rho u_t\|_{L^{2}}+\|\nabla u\|_{L^{2}}+\|\nabla u\|^{\frac{3}{2}}_{L^2}\|\nabla u\|^{\frac{1}{2}}_{L^2}.
		\end{aligned}
		\end{equation}
		Thus, by virtue of Young's inequality, we get (\ref{dd36}).
		\par Similarly, we rewrite $(\ref{a12})_2$ as follows,
		\begin{equation}\label{d37}
		\begin{aligned}
		-\operatorname{div}(\mu(\rho)\mathbb{D} w)+\nabla \pi=&-\rho\partial_{t} w-\rho(v+w)\cdot\nabla w\\&+(1-\rho)(\partial_{t}v+v\cdot\nabla v)-\rho w\cdot\nabla v+ \operatorname{div}\left(2(\mu(\rho)-1)\mathbb{D}v\right),
		\end{aligned}
		\end{equation}
		from which and using Lemma \ref{Lemma-5.1}, we get that
		\begin{equation}
		\begin{aligned}
		\|\nabla^{2}w\|_{L^{2}}&\lesssim \|\nabla w\|_{L^{2}}+\|\rho\partial_{t} w\|_{L^{2}}+\|\rho(v+w)\cdot\nabla w\|_{L^{2}}\\ &\quad+\|\rho w\cdot\nabla v\|_{L^{2}}+\|(1-\rho)(\partial_{t}v+v\cdot\nabla v)\|_{L^{2}}\\&\quad+\|\operatorname{div}\left(2(\mu(\rho)-1)\mathbb{D}v\right)\|_{L^{2}}.
		\end{aligned}
		\end{equation}
		Thanks to the Gagliardo-Nirenberg inequality, one has
		\begin{equation}\label{d38}
		\begin{aligned}
		\|\nabla^{2}w\|_{L^{2}}
		&\lesssim \|\nabla w\|_{L^{2}}+\|\rho\partial_{t} w\|_{L^{2}}+\|\nabla w\|^{\frac{3}{2}}_{L^2}\|\nabla^{2} w\|^{\frac{1}{2}}_{L^{2}}\\
		&\quad+\|u_0\|_{\dot{B}^{\frac{1}{2}}_{2,1}}\|\nabla w\|_{L^{2}}+\|1-\rho_0\|_{L^6}\left\{\|\partial_{t} v\|_{L^{3}}+\|v\|_{L^\infty}\|\nabla v\|_{L^{3}} \right\}\\
		&\quad+\|\mu(\rho_0)-1\|_{L^6}\|\Delta v\|_{L^3}+\|\nabla v\|_{L^{\frac{2q}{q-2}}}\|\nabla \mu(\rho)\|_{L^q}.
		\end{aligned}
		\end{equation}
		By Young's inequality, for $s_1\in[\frac{3}{2},\frac{5}{2}],$ we can deduce 
		$$ \|\nabla^{2} w\|_{L^{2}}\lesssim \|\rho\partial_{t} w\|_{L^{2}}+\|\nabla w\|_{L^{2}}+\|\nabla w\|^{3}_{L^{2}}+\|v\|_{\dot{B}^{s_1}_{2,1}}.$$
	\end{proof}
	
	\subsection{$L^2$ estimate of $u$}  
	\begin{Lemma}\label{Lemma-u1}
		Suppose $(\rho,u,\pi)$ is the unique local strong solution to $(\ref{a1})$ satisfying $(\ref{a7})$, then it holds that
		\begin{equation}\label{dd38}
		\sup_{t\in[t_0,T]}\int_{\mathbb{R}^3}\rho|u|^{2}dx+\int^{T}_{t_0}\int_{\mathbb{R}^3}|\nabla u|^{2}dxdt\leq C\|u(t_0)\|^{2}_{L^{2}},
		\end{equation}
		where $C$ depends on $\overline{\mu}$, $\underline{\mu}$, $M$, $\|\rho_0-1\|_{B^{\frac{3}{2}}_{2,1}}$ and $\|u_0\|_{\dot{H}^{-2\delta}}$.
	\end{Lemma}
	\begin{proof}
		Taking the $L^{2}$ inner product of $(\ref{a1})_{2}$ with $u$ and using the fact $\operatorname{div} u=0,$ we obtain
		\begin{equation}\label{d35}
		\frac{1}{2}\frac{d}{dt}\|\sqrt{\rho}u\|^{2}_{L^{2}}+2\int_{\mathbb{R}^{3}}\mu(\rho)\mathbb{D}u:\mathbb{D}udx=0.
		\end{equation}
		Integrating in time over~$[t_{0}, t]$~yields
		\begin{equation}
		\|\sqrt{\rho}u\|^{2}_{L^{\infty}([t_{0},T];L^{2})}+\|\nabla u\|^{2}_{L^{2}([t_{0},T];L^{2})}\leq C\|u(t_0)\|^{2}_{L^{2}}.
		\end{equation}
		This completes the proof of Lemma \ref{Lemma-u1}.
	\end{proof}
	
	\subsection{$\dot{H}^1$ estimate of $u$} 
	
	If we do the $\dot{H}^1$-normal estimate of $u$ directly, then the estimate cannot be closed, so here we use the idea of decomposition, i.e., $u=v+w$, so that we get not only the smallness for the $L^\infty([t_0,\infty);H^1)$ estimate of $w$, but also the smallness for the $L^\infty([t_0,\infty);L^2)$ estimate of $\nabla u$. More deeply, the smallness of the $H^1$ norm estimate of $w$ plays a crucial role in later estimating the smallness for the $L^{1}([t_0,\infty);L^\infty)$ estimate of $\nabla u$. As a convention, we will always denote $s_1\in[\frac{3}{2},\frac{5}{2}]$ in the rest of this section. For simplicity, in what follows we will only present the \emph{a priori} estimates for sufficiently smooth solutions of (\ref{dd2}) on $[0,T^{*}[.$

	\begin{Lemma}\label{Lemma-5.2}
		Suppose $(\rho,w,\pi)$ is the unique local strong solution to $(\ref{a12})$ satisfying $(\ref{a7}).$ Then under the assumptions of Proposition \ref{Proposition-5.5}, we have 
		\begin{equation}\label{d41}
		\sup_{t\in[t_0,T]}\int_{\mathbb{R}^3}|w|^{2}dx+\int^{T}_{t_0}\int_{\mathbb{R}^3}|\nabla w|^{2}dxdt\leq C\|u_{0}\|^2_{\dot{B}^{\frac{1}{2}}_{2,1}},
		\end{equation}
		where $C$ being independent of $t$.
	\end{Lemma}	
	\begin{proof}
		Firstly thanks to $\underline{\rho}<\rho_0<\overline{\rho}$, one deduces from the transport equation of (\ref{a1}) that
		$$\underline{\rho}\leq \rho(x,t)\leq \overline{\rho},$$
		from which, we get by using standard energy estimate to the $w$ equation of (\ref{a12}) that
		$$
		\begin{aligned}
		\frac{1}{2}\frac{d}{dt}\|\sqrt{\rho}w&\|^{2}_{L^2}+\|\nabla w\|^{2}_{L^2}
		\leq \left|\int_{\mathbb{R}^{3}}(1-\rho)(\partial_{t}v+v\cdot\nabla v)\cdot wdx \right|\\&+\left| \int_{\mathbb{R}^{3}}\rho w \cdot\nabla v\cdot wdx\right|+\left| \int_{\mathbb{R}^{3}}2(\mu(\rho)-1)\mathbb{D}v:\mathbb{D}wdx\right|. 
		\end{aligned}          
		$$
		Thus, we deduce that
		\begin{equation}\label{d42}
		\begin{aligned}
		& \frac{1}{2}\frac{d}{dt}\|\sqrt{\rho}w\|^{2}_{L^2}+\|\nabla w\|^{2}_{L^2}\\
		\leq &\|1-\rho_0\|_{L^2} \|\partial_{t}v+v\cdot\nabla v\|_{L^3}\|w\|_{L^6}\\
		&+\|\sqrt{\rho}w\|^2_{L^2}\|\nabla v\|_{L^\infty}+\|\nabla w\|_{L^2}\|1-\rho_0\|_{L^6}\|\nabla v\|_{L^3}\\
		\leq &\frac{1}{2}\|\nabla w\|^2_{L^2}+C\|\sqrt{\rho}w\|^2_{L^2}\|\nabla v\|_{L^\infty}+C\|v\|^{2}_{\dot{B}^{s_1}_{2,1}},
		\end{aligned}          
		\end{equation}
		where we use $\|\partial_{t} v\|_{L^{3}}\leq \|v\|_{L^3}\|\nabla v\|_{L^\infty}+\|\Delta v\|_{L^3}\leq C\|v\|_{\dot{B}^{\frac{5}{2}}_{2,1}}.$ Hence, one has
		\begin{equation}\label{d43}
		\begin{aligned}
		\frac{d}{dt}\|\sqrt{\rho}w\|^{2}_{L^2}+\|\nabla w\|^{2}_{L^2}
		&\leq C\|\nabla v\|_{L^\infty}\|\sqrt{\rho}w\|^2_{L^2}+C\|v\|^2_{\dot{B}^{s_1}_{2,1}}.
		\end{aligned}
		\end{equation}
		Integrating in time over~$[t_{0},T]$~yields
		$$
		\begin{aligned}
		\|w\|^{2}_{L^{\infty}([t_{0},T];L^{2})}+\|\nabla w\|^{2}_{L^{2}([t_{0},T];L^{2})}\leq C \|u_{0}\|^2_{\dot{B}^{\frac{1}{2}}_{2,1}}.
		\end{aligned}
		$$
		This completes the proof of Lemma \ref{Lemma-5.2}.
	\end{proof}
	
	\begin{Lemma}\label{Lemma-5.3}
		Suppose $(\rho,w,\pi)$ is the unique local strong solution to $(\ref{a12})$ satisfying $(\ref{a7}).$ Then it holds that
		\begin{equation}\label{d46}
		\begin{aligned}
		\int^{T}_{t_{0}}\int_{\mathbb{R}^{3}}\rho|w_{t}|^{2}dxdt+\sup_{t\in[t_{0},T]}\int_{\mathbb{R}^{3}}|\nabla w|^{2} dx\leq 2C\|u_{0}\|^{2}_{\dot{B}^{\frac{1}{2}}_{2,1}},
		\end{aligned}
		\end{equation}
		where $C$ being independent of $t$.
	\end{Lemma}
	\begin{proof} 
		Multiplying the momentum equations $(\ref{a12})_2$ by $w_{t}$ and integrating over $\mathbb{R}^{3}$ yield
		\begin{equation}\label{d47}
		\begin{aligned}
		&\int_{\mathbb{R}^{3}}\rho| w_{t}|^{2}dx+\frac{d}{dt}\int_{\mathbb{R}^{3}}\mu(\rho)\mathbb{D}w:\mathbb{D}w dx\\
		\leq &\left|\int_{\mathbb{R}^{3}}(1-\rho)(\partial_{t}v+v\cdot\nabla v)\cdot \partial_{t}w dx \right|+\left|\int_{\mathbb{R}^{3}}\rho (w\cdot\nabla v)\cdot w_{t}dx \right|\\&+\left|\int_{\mathbb{R}^{3}}\operatorname{div}(2(\mu(\rho)-1)\mathbb{D}v)\cdot \partial_{t}w dx\right|+\left|\int_{\mathbb{R}^{3}}\partial_{t}\mu(\rho)\mathbb{D}w:\mathbb{D}w dx\right|\\&+\left|\int_{\mathbb{R}^{3}}\rho (v+w)\cdot\nabla w\cdot w_{t} dx\right|.
		\end{aligned}
		\end{equation}
		\par Applying Gagliardo-Nirenberg inequality,
		$$\begin{aligned}
		\left|\int_{\mathbb{R}^{3}}(1-\rho)(\partial_{t}v+v\cdot\nabla v)\cdot\partial_{t}w dx \right|&\leq
		\| a_0\|_{L^6}\|\partial_{t}v+v\cdot\nabla v\|_{L^{3}}\|\sqrt{\rho}w_t\|_{L^2}\\
		&\leq \frac{1}{8}\|\sqrt{\rho}w_t\|^{2}_{L^2}+C\|v\|^{2}_{\dot{B}^{s_1}_{2,1}},
		\end{aligned}$$
		and
		$$\begin{aligned}
		\left|\int_{\mathbb{R}^{3}}\rho (w\cdot\nabla v)\cdot w_{t}dx \right|
		&\leq\|\sqrt{\rho}w_t\|_{L^{2}}\|w\|_{L^{6}}\|\nabla v\|_{L^{3}}\\
		&\leq \frac{1}{8}\|\sqrt{\rho}w_t\|^{2}_{L^2}+C\|\nabla w\|^{2}_{L^2}\|\nabla v\|^{2}_{L^3}.
		\end{aligned}$$
		Similarly,
		$$\begin{aligned}
		\left|\int_{\mathbb{R}^{3}}\operatorname{div}(2(\mu(\rho)-1)\mathbb{D}v)\cdot\partial_{t}w dx\right|
		\leq &\|(\mu(\rho)-1)\Delta v+\nabla\mu(\rho)\cdot\mathbb{D}v\|_{L^2}\|\sqrt{\rho}w_t\|_{L^2}\\
		\leq &\left\{\|\mu(\rho)-1\|_{L^6}\|\Delta v\|_{L^3}+\|\nabla v\|_{L^\frac{2q}{q-2}}\|\nabla\mu(\rho)\|_{L^q}\right\}\|\sqrt{\rho}w_t\|_{L^2}\\
		\leq& \frac{1}{8}\|\sqrt{\rho}w_t\|^{2}_{L^2}+C(\|\rho_0-1\|_{L^6},M)\|v\|^{2}_{\dot{B}^{s_1}_{2,1}},
		\end{aligned}$$
		where we require $q\in(3,6)$, and along the same way,
		$$\begin{aligned}
		&\left|\int_{\mathbb{R}^{3}}\partial_{t}\mu(\rho)\mathbb{D}w:\mathbb{D}w dx\right|\leq\left|\int_{\mathbb{R}^{3}}(v+w)\cdot\nabla\mu(\rho)\mathbb{D}w:\mathbb{D}w dx\right|\\
		\leq& \|\nabla \mu(\rho)\|_{L^q}\|\nabla w\|_{L^\frac{2q}{q-2}}\|\nabla w\|_{L^2}\left( \|w\|_{L^\infty}+\|v\|_{L^\infty}\right)\\
		\leq& C\|\nabla \mu(\rho)\|_{L^q}\|\nabla w\|^{\frac{q-3}{q}}_{L^2}\|\nabla^2 w\|^{\frac{3}{q}}_{L^2}\|\nabla w\|_{L^2}\left( \|w\|_{L^\infty}+\|v\|_{L^\infty}\right)\\
		\leq& C\|\nabla \mu(\rho)\|_{L^q}(\|\nabla w\|_{L^2}+\|\nabla^2 w\|_{L^2})\|\nabla w\|_{L^2}\left( \|w\|_{L^\infty}+\|v\|_{L^\infty}\right)\\
		\leq& C(M)\|\nabla w\|^{\frac{5}{2}}_{L^2}\|\nabla^{2}w\|^{\frac{1}{2}}_{L^2}+C(M)\|\nabla w\|^{\frac{3}{2}}_{L^2}\|\nabla^{2}w\|^{\frac{3}{2}}_{L^2}\\
		&\quad+C(M)\|\nabla w\|^{2}_{L^2}\|v\|_{L^\infty}+C(M)\|\nabla w\|_{L^2}\|\nabla^{2}w\|_{L^2}\|v\|_{L^\infty}.
		\end{aligned}$$
		Here we have used the fact that
		$$\partial_{t}\mu(\rho)+(v+w)\cdot\nabla \mu(\rho)=0,$$
		which is a consequence of mass equation and the fact $\operatorname{div} u=0.$ Notice that
		$$\begin{aligned}
		&\left|\int_{\mathbb{R}^{3}}\rho (v+w)\cdot\nabla w\cdot w_{t} dx\right|
		\\\leq &\|\sqrt{\rho}w_t\|_{L^2}\left\{ \|\nabla w\|_{L^3}\|w\|_{L^6}+\|\nabla w\|_{L^2}\|v\|_{L^\infty}\right\}\\
		\leq &\|\sqrt{\rho}w_t\|_{L^2}\|\nabla w\|^{\frac{3}{2}}_{L^2}\|\nabla^{2}w\|^{\frac{1}{2}}_{L^2}+\|\sqrt{\rho}w_t\|_{L^2}\|\nabla w\|_{L^2}\|v\|_{L^\infty}.
		\end{aligned}$$
		Hence, by Young's inequality and Corollary \ref{Corollary-5.1},
		\begin{equation}\label{d48}
		\begin{aligned}
		&\int_{\mathbb{R}^{3}}\rho| w_{t}|^{2}dx+\frac{d}{dt}\int_{\mathbb{R}^{3}}\mu(\rho)\mathbb{D}w:\mathbb{D}w dx\\
		\leq& \frac{1}{2}\|\sqrt{\rho}w_t\|^{2}_{L^2}+\|\nabla w\|^{\frac{3}{2}}_{L^2}\left\{\|\rho\partial_{t} w\|_{L^{2}}+\|\nabla w\|_{L^{2}}+\|\nabla w\|^{3}_{L^{2}}+\|v\|_{\dot{B}^{s_1}_{2,1}}\right\}^{\frac{3}{2}}\\&+C\|\nabla w\|_{L^2}\|v\|_{L^\infty}\left\{\|\rho\partial_{t} w\|_{L^{2}}+\|\nabla w\|_{L^{2}}+\|\nabla w\|^{3}_{L^{2}}+\|v\|_{\dot{B}^{s_1}_{2,1}}\right\}\\
		&+C\|\sqrt{\rho}w_t\|_{L^2}\|\nabla w\|^{\frac{3}{2}}_{L^2}\left\{\|\rho\partial_{t} w\|_{L^{2}}+\|\nabla w\|_{L^{2}}+\|\nabla w\|^{3}_{L^{2}}+\|v\|_{\dot{B}^{s_1}_{2,1}}\right\}^{\frac{1}{2}}\\
		&+C\|\nabla w\|^{2}_{L^2}\| v\|^{2}_{\dot{B}^{s_1}_{2,1}}+\|\nabla w\|^3_{L^2}+C\|v\|^{2}_{\dot{B}^{s_1}_{2,1}},
		\end{aligned}
		\end{equation}
		which yields that
		\begin{equation}\label{d49}
		\begin{aligned}
		&\int_{\mathbb{R}^{3}}\rho| w_{t}|^{2}dx+\frac{d}{dt}\int_{\mathbb{R}^{3}}\mu(\rho)\mathbb{D}w:\mathbb{D}w dx\\
		\leq &\frac{7}{8}\|\sqrt{\rho}w_t\|^{2}_{L^2}+C\|\nabla w\|^{6}_{L^2}+C\|\nabla w\|^{4}_{L^2}+C\|\nabla w\|^{3}_{L^2}\\&+C\|v\|^{2}_{\dot{B}^{s_1}_{2,1}}+C\|\nabla w\|^{2}_{L^2}\| v\|^{2}_{\dot{B}^{s_1}_{2,1}}.
		\end{aligned}
		\end{equation} 
		Thus, we obatin
		\begin{equation}\label{dd50}
		\begin{aligned}
		\int_{\mathbb{R}^{3}}\rho| w_{t}|^{2}dx+\frac{d}{dt}\int_{\mathbb{R}^{3}}\mu(\rho)\mathbb{D}w:\mathbb{D}w dx
		\leq& C\|\nabla w\|^{6}_{L^2}+C\|\nabla w\|^{2}_{L^2}+C\|v\|^{2}_{\dot{B}^{s_1}_{2,1}}.
		\end{aligned}
		\end{equation}
		Integrating with respect to time on $[t_{0}, T]$ gives
		$$
		\begin{aligned}
		&\int^{T}_{t_{0}}\int_{\mathbb{R}^{3}}\rho| w_{t}|^{2}dxdt+\sup_{t\in[t_{0},T]}\int_{\mathbb{R}^{3}}|\nabla w|^{2} dx\\
		\leq &C\int^{T}_{t_{0}}\|\nabla w\|^{6}_{L^{2}}dt+C\int^{T}_{t_{0}}\|\nabla w\|^{2}_{L^{2}}dt+C\int^{T}_{t_{0}}\|v\|^{2}_{\dot{B}^{s_1}_{2,1}}dt.
		\end{aligned}
		$$
		Applying Gronwall's inequality,
		$$
		\begin{aligned}
		&\int^{T}_{t_{0}}\int_{\mathbb{R}^{3}}\rho| w_{t}|^{2}dxdt+\sup_{t\in[t_{0},T]}\int_{\mathbb{R}^{3}}|\nabla w|^{2} dx
		\leq C\|u_0\|^{2}_{\dot{B}^{\frac{1}{2}}_{2,1}}\cdot\exp\{ C\int^{T}_{t_{0}}\|\nabla w\|^{4}_{L^{2}}dt\}.
		\end{aligned}
		$$
		Such that if
		\begin{equation}\label{d50}
		\|u_{0}\|_{\dot{B}^{\frac{1}{2}}_{2,1}}\leq \varepsilon_{1}\quad{\rm and}\quad	\sup_{t\in[t_{0},T]}\|\nabla w\|^{2}_{L^{2}}\leq  4C\|u_{0}\|^{2}_{\dot{B}^{\frac{1}{2}}_{2,1}}\leq 1,
		\end{equation}
		then
		\begin{equation}\label{d51}
		\begin{aligned}
		\int^{T}_{t_{0}}\|\nabla w\|^{4}_{L^{2}}dt&\leq \sup_{t\in[t_{0},T]}\|\nabla w\|^{2}_{L^{2}}\cdot\int^{T}_{t_{0}}\|\nabla w\|^{2}_{L^{2}}dt\\&
		\leq 4C\|u_{0}\|^{4}_{\dot{B}^{\frac{1}{2}}_{2,1}}\leq \|u_{0}\|^{2}_{\dot{B}^{\frac{1}{2}}_{2,1}}.
		\end{aligned}
		\end{equation}
		Hence, we arrive at
		\begin{equation}\label{d52}
		\begin{aligned}
		\int^{T}_{t_{0}}\int_{\mathbb{R}^{3}}\rho| w_{t}|^{2}dxdt+\sup_{t\in[t_{0},T]}\int_{\mathbb{R}^{3}}|\nabla w|^{2} dx\leq 2C\|u_{0}\|^{2}_{\dot{B}^{\frac{1}{2}}_{2,1}}.
		\end{aligned}
		\end{equation}
		Choose some small positive constant $	\varepsilon_{1}=\min\{\sqrt{\frac{1}{4C}}, \sqrt{\frac{\ln2}{C}}\}$, which is clear that (\ref{d52}) holds, provided (\ref{d50}) holds.
	\end{proof}
	
	\begin{Lemma}\label{Lemma-u}
		Suppose $(\rho,u,\pi)$ is the unique local strong solution to $(\ref{a1})$ satisfying $(\ref{a7}).$ Then it holds that
		\begin{equation}\label{u}
		\begin{aligned}
		\int^{T}_{t_{0}}\int_{\mathbb{R}^{3}}|u_{t}|^{2}dxdt+\sup_{t\in[t_{0},T]}\int_{\mathbb{R}^{3}}|\nabla u|^{2} dx\leq 2C\|u_{0}\|^{2}_{\dot{B}^{\frac{1}{2}}_{2,1}},
		\end{aligned}
		\end{equation}
		where $C$ being independent of $t$.
	\end{Lemma}
	\begin{proof}
		Multiplying the classical Navier-Stokes equation $(\ref{a11})_1$ by $v_{t}$ and integrating over $\mathbb{R}^{3}$ yield
		\begin{equation}\label{u1}
		\begin{aligned}
		\int_{\mathbb{R}^{3}}| v_{t}|^{2}dx+\frac{1}{2}\frac{d}{dt}\int_{\mathbb{R}^{3}}|\nabla v|^{2}dx
		&=-\int_{\mathbb{R}^3} v\cdot\nabla v\cdot v_t dx\\
		&\leq \frac{1}{2}\|v_t\|^2_{L^2}+C\|\nabla v\|^2_{L^2}\|v\|^2_{L^\infty},
		\end{aligned}
		\end{equation}
		from which and Proposition $\ref{Proposition-5.1}$, we arrive at
		\begin{equation}
		\int^{T}_{t_{0}}\int_{\mathbb{R}^{3}}|v_{t}|^{2}dxdt+\sup_{t\in[t_{0},T]}\int_{\mathbb{R}^{3}}|\nabla v|^{2} dx\leq 2C\|u_{0}\|^{2}_{\dot{B}^{\frac{1}{2}}_{2,1}},
		\end{equation}
		By virtue of $u\stackrel{\text { def }}{=}v+w$ and (\ref{d46}), we completes the proof of Lemma \ref{Lemma-u}.
	\end{proof}
	\subsection{The \emph{a priori} time-weighted estimates to (\ref{a1})}  
	
	\begin{Lemma}\label{Lemma-5.4}
		Suppose $(\rho,u,\pi)$ is the unique local strong solution to $(\ref{a1})$ satisfying $(\ref{a7}).$ Then under the assumptions of Proposition \ref{Proposition-5.5}, we have 
		\begin{equation}\label{d53}
		\begin{aligned}
		\int^{T}_{t_{0}}\int_{\mathbb{R}^{3}}t|u_{t}|^{2}dxdt+\sup_{t\in[t_{0},T]}\int_{\mathbb{R}^{3}}t|\nabla u|^{2} dx\leq  \exp\{C\mathcal{H}_{0}\},
		\end{aligned}
		\end{equation}
		where $\mathcal{H}_{0}$ is given by (\ref{h}) and $C$ depends on $\underline{\mu}$, $\overline{\mu}$, $M$, $\|\rho_0-1\|_{B^{\frac{3}{2}}_{2,1}}$ and $\|u_0\|_{\dot{H}^{-2\delta}}$.
	\end{Lemma}
	\begin{proof} 
		Multiplying the momentum equations by $u_{t}$ and integrating over $\mathbb{R}^{3}$ yield
		\begin{equation}\label{dd52}
		\begin{aligned}
		&\int_{\mathbb{R}^{3}}\rho| u_{t}|^{2}dx+\frac{d}{dt}\int_{\mathbb{R}^{3}}\mu(\rho)\mathbb{D}u:\mathbb{D}u dx\\
		\leq&|\int_{\mathbb{R}^{3}}\rho u\cdot\nabla u\cdot u_{t}dx|+\int_{\mathbb{R}^{3}}|\nabla\mu(\rho)|\cdot|u|\cdot|\nabla u|^{2} dx.
		\end{aligned}
		\end{equation}
		Along the same way as (\ref{d49}) and using Corollary \ref{Corollary-5.1} to estimate each term on the right hand of (\ref{dd52}), we obtain
		\begin{equation}\label{dd53}
		\begin{aligned}
		\int_{\mathbb{R}^{3}}\rho|u_{t}|^{2}dx+\frac{d}{dt}\int_{\mathbb{R}^{3}}\mu(\rho)\mathbb{D}u:\mathbb{D}u dx \leq\frac{3}{8}\|\sqrt{\rho}u_{t}\|^{2}_{L^{2}}+C\|\nabla u\|^{3}_{L^{2}}+C\|\nabla u\|^{4}_{L^{2}}+C\|\nabla u\|^{6}_{L^{2}},
		\end{aligned}
		\end{equation}
		Multiplying (\ref{dd53}) by $t,$ as shown in the last proof, one has 
		\begin{equation}\label{d54}
		\begin{aligned}
		&\int^{T}_{t_{0}}t\|u_{t}\|^{2}_{L^{2}}dt+\sup_{t\in[t_{0},T]}t\|\nabla u\|^{2}_{L^{2}}\leq C\|\nabla u(t_0)\|^2_{L^2}+C\int^{T}_{t_{0}}\|\nabla u\|^{2}_{L^{2}}dt
		\\ &\quad+ C\int^{T}_{t_{0}}t\|\nabla u\|^{3}_{L^2}dt
		+C\int^{T}_{t_{0}}t\|\nabla u\|^{4}_{L^{2}}dt+C\int^{T}_{t_{0}}t\|\nabla u\|^{6}_{L^{2}}dt.
		\end{aligned}
		\end{equation}
		According to Corollary \ref{Corollary-5.2} and $(\ref{d1})$, we get for $\delta_{-}>\frac{1}{2},$ 
		\begin{equation}\label{d55}
		\int^{T}_{t_{0}}\|\nabla u\|_{L^2}dt\leq\left(\int^{T}_{t_{0}}\|t^{\delta_{-}}\nabla u\|^{2}_{L^2}dt\right)^{\frac{1}{2}}\left(\int^{T}_{t_{0}}t^{-2\delta_{-}}dt\right)^{\frac{1}{2}}  \leq C\mathcal{H}_{0}.
		\end{equation}   
		Applying Gronwall's inequality,
		\begin{equation}\label{d56}
		\begin{aligned}
		\int^{T}_{t_{0}}t\|u_{t}\|^{2}_{L^{2}}dt+\sup_{t\in[t_{0},T]}t\|\nabla u\|^{2}_{L^{2}}
		\leq  C\| u(t_{0})\|^{2}_{H^{1}}\exp\{C\mathcal{H}_{0}\}.
		\end{aligned}
		\end{equation}
		This completes the proof of Lemma \ref{Lemma-5.4}.
	\end{proof}
	\begin{Lemma}\label{Lemma-5.5}
		Suppose $(\rho,u,\pi)$ is the unique local strong solution to $(\ref{a1})$ satisfying $(\ref{a7}).$ Then under the assumptions of Proposition \ref{Proposition-5.5}, we have 
		\begin{equation}\label{d57}
		\begin{aligned}
		\sup_{t\in[t_{0},T]}t\|u_{t}\|^{2}_{L^{2}}+\int^{T}_{t_{0}}t\|\nabla u_{t}\|^{2}_{L^{2}}dt\leq \exp\{C\mathcal{H}_{0}\},
		\end{aligned}
		\end{equation}
		and
		\begin{equation}\label{d58}
		\begin{aligned}
		\sup_{t\in[t_{0},T]}t^2\|u_{t}\|^{2}_{L^{2}}+\int^{T}_{t_{0}}t^{2}\|\nabla u_{t}\|^{2}_{L^{2}}dt\leq \exp\{C\mathcal{H}_{0}\},
		\end{aligned}
		\end{equation}
		where $\mathcal{H}_{0}$ is given by (\ref{h}) and $C$ depends on $\underline{\mu}$, $\overline{\mu}$, $M$, $\|\rho_0-1\|_{B^{\frac{3}{2}}_{2,1}}$ and $\|u_0\|_{\dot{H}^{-2\delta}}$.
	\end{Lemma}
	\begin{proof}
		Take $t$-derivative of the momentum equations,
		\begin{equation}\label{d59}
		\begin{aligned}
		\rho\partial_{tt}u+\rho u\cdot\nabla u_{t}-\operatorname{div}(2\mu(\rho)\mathbb{D}u_{t})+\nabla\pi_{t}=-\rho_t u_t- \left(\rho u\right)_t\cdot\nabla u+
		\operatorname{div}(2\partial_{t}\mu(\rho)\mathbb{D}u).
		\end{aligned}
		\end{equation}
		Multiplying (\ref{d59}) by $u_{t}$ and integrating over $\mathbb{R}^{3}$, we get after integration by parts that
		\begin{equation}\label{d60}
		\begin{aligned}
		&\frac{1}{2}\frac{d}{dt}\int_{\mathbb{R}^{3}}\rho|u_{t}|^{2}dx+2\int_{\mathbb{R}^{3}}\mu(\rho)\mathbb{D}u_{t}:\mathbb{D}u_{t}dx\\=&-\int_{\mathbb{R}^{3}}\rho_t u_t\cdot u_{t}dx-\int_{\mathbb{R}^{3}} \left(\rho u\right)_t\cdot\nabla u\cdot u_{t}dx
		-2\int_{\mathbb{R}^{3}}\partial_{t}\mu(\rho)\mathbb{D}u\cdot\nabla u_{t}dx\\\stackrel{\text { def }}{=}&\sum_{i=1}^{3} J_{i}.
		\end{aligned}
		\end{equation} 
		\par Now, we will use the Gagliardo-Nirenberg inequality to estimate each term on the right hand of (\ref{d60}). First, thanks to the mass equation, one has
		\begin{equation}\label{j1}
		\begin{aligned}
		J_1=-2\int_{\mathbb{R}^{3}}\rho u\cdot\nabla u_t\cdot u_{t}dx&\leq C\|u\|_{L^6}\|\nabla u_t\|_{L^2}\|u_t\|_{L^3}\\
		&\leq C\|\nabla u\|_{L^2}\|u_t\|_{L^2}^{\frac{1}{2}}\|\nabla u_t\|^{\frac{3}{2}}_{L^2}\\
		&\leq \frac{1}{8} \underline{\mu} \|\nabla u_t\|^{2}_{L^2}+C\|u_t\|^{2}_{L^2}\|\nabla u\|^{4}_{L^2}.
		\end{aligned}
		\end{equation} 
		Taking into account the mass equation again, we arrive at
		$$\begin{aligned}
		J_2&=-\int_{\mathbb{R}^{3}}\rho_t u \cdot\nabla u|u_{t}dx-\int_{\mathbb{R}^{3}}\rho u_t\cdot\nabla u|u_{t}dx\\
		& \leq  \int_{\mathbb{R}^{3}} \rho|u| \cdot|\nabla u|^2 \cdot\left|u_t\right| dx+C  \int_{\mathbb{R}^{3}} \rho|u|^2 \cdot\left|\nabla^2 u\right| \cdot\left|u_t\right| d x \\
		& \quad+ \int_{\mathbb{R}^{3}}\rho|u|^2 \cdot|\nabla u| \cdot\left|\nabla u_t\right| d x+ \int_{\mathbb{R}^{3}} \rho\left|u_t\right|^2 \cdot|\nabla u| d x.
		\end{aligned}$$
		Hence, it follows from Sobolev embedding inequality, Gagliardo–Nirenberg inequality, that
		$$
		\begin{aligned}
		\int_{\mathbb{R}^{3}} \rho|u| \cdot|\nabla u|^2 \cdot\left|u_t\right| d x & \leq C\left\|u_t\right\|_{L^6} \|u\|_{L^6} \|\nabla u\|_{L^3}^2 \\
		& \leq C \left\|\nabla u_t\right\|_{L^2} \|\nabla u\|_{L^2}^2 \|\nabla^2 u\|_{L^2} \\
		& \leq \frac{1}{8} \underline{\mu}\left\|\nabla u_t\right\|_{L^2}^2+C\|\nabla u\|_{L^2}^4\left\|\nabla^2 u\right\|_{L^2}^2.
		\end{aligned}
		$$
		Similarly, it holds that
		$$
		\begin{aligned}
		\int_{\mathbb{R}^{3}} \rho|u|^2 \cdot\left|\nabla^2 u\right| \cdot\left|u_t\right| d x & \leq C \left\|u_t\right\|_{L^6} \left\|\nabla^2 u\right\|_{L^2} \|u\|_{L^6}^2 \\
		& \leq \frac{1}{8} \underline{\mu}\left\|\nabla u_t\right\|_{L^2}^2+C\left\|\nabla^2 u\right\|^2_{L^2} \|\nabla u\|_{L^2}^4,
		\end{aligned}
		$$
		and
		$$
		\begin{aligned}
		\int_{\mathbb{R}^{3}}\rho|u|^2 \cdot|\nabla u| \cdot\left|\nabla u_t\right| d x & \leq C \left\|\nabla u_t\right\|_{L^2} \|\nabla u\|_{L^6} \|u\|_{L^6}^2 \\
		& \leq \frac{1}{8} \underline{\mu} \left\|\nabla u_t\right\|_{L^2}^2+ C\left\|\nabla^2 u\right\|_{L^2}^2 \|\nabla u\|_{L^2}^4.
		\end{aligned}
		$$
		Owing to Sobolev embedding inequality,
		$$\begin{aligned}
		\int_{\mathbb{R}^{3}} \rho\left|u_t\right|^2 \cdot|\nabla u| d x&\leq C\left\|u_t\right\|^2_{L^4} \|\nabla u\|_{L^2} \\
		&\leq  C\left\|u_t\right\|^{\frac{1}{2}}_{L^2} \left\|\nabla u_t\right\|^{\frac{3}{2}}_{L^2}\|\nabla u\|_{L^2}\\
		&\leq \frac{1}{8} \underline{\mu} \left\|\nabla u_t\right\|_{L^2}^2+C\|u_t\|^2_{L^2}\|\nabla u\|^4_{L^2}.
		\end{aligned}$$
		Thanks to Corollary \ref{Corollary-5.1}, we deduce that
		\begin{equation}\label{j2}
		\begin{aligned}
		|J_2|&\leq \frac{1}{2} \underline{\mu}\left\|\nabla u_t\right\|_{L^2}^2+C\|\nabla u\|_{L^2}^4\left\|\nabla^2 u\right\|_{L^2}^2+C\|u_t\|^2_{L^2}\|\nabla u\|^4_{L^2}\\
		&\leq \frac{1}{2} \underline{\mu}\left\|\nabla u_t\right\|_{L^2}^2 +C\|u_t\|^2_{L^2}\|\nabla u\|^4_{L^2}+C\|\nabla u\|_{L^2}^6+C\|\nabla u\|_{L^2}^{10}.
		\end{aligned}
		\end{equation}
		Along the same way, we obtain
		$$
		\begin{aligned}
		\left|J_3\right| &\leq C \int|u| \cdot|\nabla \mu(\rho)|\cdot |\mathbb{D}u| \cdot\left|\nabla u_t\right| d x \\
		& \leq C\|\nabla \mu(\rho)\|_{L^q}\|u\|_{L^{\infty}}\left\|\nabla u_t\right\|_{L^2}\|\nabla u\|_{L^{\frac{2 q}{q-2}}} \\
		& \leq C(q, M)\|u\|_{L^6}^{1 / 2}\|\nabla u\|_{L^6}^{1 / 2}\left\|\nabla u_t\right\|_{L^2}\|\nabla u\|_{L^2}^{\frac{q-3}{q}}\left\|\nabla^2 u\right\|_{L^2}^{\frac{3}{q}} \\
		& \leq \frac{1}{8} \underline{\mu}\left\|\nabla u_t\right\|_{L^2}^2+C\|\nabla u\|_{L^2}\left\|\nabla^2 u\right\|_{L^2}^3+C\|\nabla u\|_{L^2}^4,
		\end{aligned}
		$$
		thus, it follows from Corollary \ref{Corollary-5.1} that
		\begin{equation}\label{j3}
		\begin{aligned}
		|J_3|\leq \frac{1}{8} \underline{\mu}\left\|\nabla u_t\right\|_{L^2}^2+C\|\nabla u\|_{L^2}\left\|u_t\right\|_{L^2}^3+C\|\nabla u\|_{L^2}^4+C\|\nabla u\|^{10}_{L^2}.
		\end{aligned}
		\end{equation}
		\par Substituting (\ref{j1})-(\ref{j3}) into (\ref{d60}) and applying Corollary \ref{Corollary-5.1}, gives
		\begin{equation}\label{d68}
		\begin{aligned}
		\frac{d}{dt}\int_{\mathbb{R}^{3}}\rho|u_{t}|^{2}dx+&\|\nabla u_{t}\|^{2}_{L^{2}} \leq C\|u_t\|^2_{L^2}\|\nabla u\|^4_{L^2}+C\|u_t\|^4_{L^2}\\&+C\|\nabla u\|_{L^2}^4+C\|\nabla u\|_{L^2}^6+C\|\nabla u\|^{10}_{L^2}.
		\end{aligned}
		\end{equation}
		Multiplying (\ref{d68}) by $t$ gives that
		$$
		\begin{aligned}
		\frac{d}{dt}\|\sqrt{t}&\sqrt{\rho}u_{t}\|^{2}_{L^{2}}+\|\sqrt{t}\nabla u_{t}\|^{2}_{L^{2}}\lesssim \|u_{t}\|^{2}_{L^{2}}+\|\sqrt{t}u_{t}\|^{2}_{L^{2}}\|\nabla u\|^{4}_{L^{2}}\\&+ \|\sqrt{t}u_{t}\|^{2}_{L^{2}}\|u_{t}\|^{2}_{L^{2}}+\|\sqrt{t}\nabla u\|^{2}_{L^{2}}\{\|\nabla u\|^{2}_{L^{2}}+\|\nabla u\|^{4}_{L^{2}}+\|\nabla u\|^{8}_{L^{2}}\}.
		\end{aligned}
		$$
		Integrating with respect to time on $[t_{0},T]$ and using Gronwall's inequality gives
		\begin{equation}\label{d69}
		\begin{aligned}
		\sup_{t\in[t_{0},T]}t\|&u_{t}\|^{2}_{L^{2}}+\int^{T}_{t_{0}}t\|\nabla u_{t}\|^{2}_{L^{2}}dt\lesssim\{\|u_{t}(t_0)\|^{2}_{L^{2}}+\int^{T}_{t_{0}}\|u_t\|^{2}_{L^{2}}dt\\&+\int^{T}_{t_{0}}\|\sqrt{t}\nabla u\|^{2}_{L^{2}}(\|\nabla u\|^{2}_{L^{2}}+\|\nabla u\|^{4}_{L^{2}}+\|\nabla u\|^{8}_{L^{2}})dt\}\\&\quad\times\exp\left\{\int^{T}_{t_{0}}\|\nabla u\|^{4}_{L^{2}}dt+\int^{T}_{t_0}\|u_t\|^{2}_{L^{2}}dt\right\}.
		\end{aligned}
		\end{equation}
		According to Lemma \ref{Lemma-u} and \ref{Lemma-5.4},
		\begin{equation}\label{d70}
		\begin{aligned}
		&\int^{T}_{t_{0}}\|\sqrt{t}\nabla u\|^{2}_{L^{2}}\left(\|\nabla u\|^{2}_{L^{2}}+\|\nabla u\|^{4}_{L^{2}}+\|\nabla u\|^{8}_{L^{2}}\right)dt\\
		\leq &\sup_{t\in[t_{0},T]}\|\sqrt{t}\nabla u\|^{2}_{L^{2}}\cdot\left(1+\sup_{t\in[t_{0},T]}\|\nabla u\|^{2}_{L^{2}}+\sup_{t\in[t_{0},T]}\|\nabla u\|^{6}_{L^{2}}\right)\\&\quad\times\int^{T}_{t_{0}}\|\nabla u\|^{2}_{L^{2}}dt\leq \|u_0\|^2_{\dot{B}^{\frac{1}{2}}_{2,1}}\exp\{C\mathcal{H}_{0}\}.
		\end{aligned}
		\end{equation}
		Whereas taking $L^2$-norm of the $u_t$ at $t=t_0$ and using (\ref{d1}) gives rise to
		\begin{equation}\label{d71}
		\begin{aligned}
		\underline{\rho}\|u_t(t_0)\|_{L^2}
		\lesssim &\|\rho u\cdot\nabla u (t_0)\|_{L^2}+\|\operatorname{div}(2\mu(\rho)\mathbb{D}u)(t_0)\|_{L^2}\\
		\lesssim & \|u(t_0)\|_{L^6}\|\nabla u(t_0)\|_{L^3}+\|\nabla\mu(\rho)(t_0)\|_{L^q}\|\mathbb{D}u(t_0)\|_{L^{\frac{q-2}{2q}}}+\|\Delta u(t_0)\|_{L^2}\\
		\lesssim & \|\nabla u(t_0)\|^{\frac{3}{2}}_{L^2}\|\nabla^{2}u(t_0)\|^{\frac{1}{2}}_{L^2}+M\|\nabla u(t_0)\|^{\frac{q-3}{q}}_{L^{2}}\|\nabla^2 u(t_0)\|^{\frac{3}{q}}_{L^2}+\|\Delta u(t_0)\|_{L^2}\\
		\lesssim & \|u_0\|_{\dot{B}^{\frac{1}{2}}_{2,1}}.
		\end{aligned}
		\end{equation}
		Plugging (\ref{d70}) and (\ref{d71}) into (\ref{d69}), we have
		\begin{equation}\label{d72}
		\begin{aligned}
		\sup_{t\in[t_{0},T]}t\|u_{t}\|^{2}_{L^{2}}+\int^{T}_{t_{0}}t\|\nabla u_{t}\|^{2}_{L^{2}}dt\leq\|u_0\|^2_{\dot{B}^{\frac{1}{2}}_{2,1}}\exp\{C\mathcal{H}_{0}\}.
		\end{aligned}
		\end{equation}
		On the other hand, multiplying (\ref{d68}) by $t^2,$ one has
		$$
		\begin{aligned}
		\frac{d}{dt}\|t\sqrt{\rho}u_{t}&\|^{2}_{L^{2}}+\|t\nabla u_{t}\|^{2}_{L^{2}}\lesssim \|\sqrt{t}u_{t}\|^{2}_{L^{2}}+ \|tu_{t}\|^{2}_{L^{2}}\|\nabla u\|^{4}_{L^{2}}\\&+\|tu_{t}\|^{2}_{L^{2}}\|u_t\|^{2}_{L^{2}}+\|\sqrt{t}\nabla u\|^{4}_{L^{2}}\{1+\|\nabla u\|^{2}_{L^{2}}+\|\nabla u\|^{6}_{L^{2}}\}.
		\end{aligned}
		$$
		Owing to Corollary \ref{Corollary-5.2}, we get for $\delta>\frac{1}{2},$ 
		$$
		\begin{aligned}
		\int^{T}_{t_0}\|\sqrt{t}\nabla u\|^{4}_{L^{2}}dt\leq \sup_{t\in[t_0,T]}\|\sqrt{t}\nabla u\|^{2}_{L^{2}}\int^{T}_{t_0}t^{1-2\delta_{-}}\|t^{\delta_{-}}\nabla u\|^{2}_{L^{2}}dt\leq \exp\{C\mathcal{H}_{0}\}.
		\end{aligned}$$ 
		From which and Gronwall's inequality, we can deduce
		\begin{equation}\label{d74}
		\begin{aligned}
		&\sup_{t\in[t_{0},T]}t^2\|u_{t}\|^{2}_{L^{2}}+\int^{T}_{t_{0}}t^2\|\nabla u_{t}\|^{2}_{L^{2}}dt\leq \exp\{C\mathcal{H}_{0}\}.
		\end{aligned}
		\end{equation}
		This completes the proof of Lemma \ref{Lemma-5.5}.
	\end{proof}
	
	\subsection{The $L^{1}([t_0,T];L^\infty(\mathbb{R}^{3}))$ estimate for $\nabla u$}
	To close the $L^\infty([t_0,T];L^q(\mathbb{R}^{3}))$ estimate of $\nabla \mu(\rho)$ in this section, we need the $L^{1}([t_0,T];L^\infty(\mathbb{R}^{3})))$ estimate of $\nabla u$ to be small enough, i.e:
	\begin{Lemma}\label{Lemma-5.6}
		Let $p\in(2,\frac{6\delta_{-}}{1+\delta_{-}})$ for $\delta\in(\frac{1}{2},\frac{3}{4})$ and $\theta_1=\frac{(r-3)p}{3r-3p+pr}$. Assume that $(\rho,u,\pi)$ is the unique local strong solution to $(\ref{a1})$ satisfying $(\ref{a7}).$ Then under the assumptions of Proposition \ref{Proposition-5.5}, we have 
		\begin{equation}\label{d75}
		\begin{aligned}
		\|\nabla u\|_{L^{1}([t_{0},T];L^\infty)}\lesssim\|u_{0}\|^{\frac{3\theta_1(p-2)}{2p}}_{\dot{B}^{\frac{1}{2}}_{2,1}}\exp\{C\mathcal{H}_{0}\},
		\end{aligned}
		\end{equation}
		where $\mathcal{H}_{0}$ is given by (\ref{h}) and $C$ depends on $\underline{\mu}$, $\overline{\mu}$, $M$, $\|\rho_0-1\|_{B^{\frac{3}{2}}_{2,1}}$ and $\|u_0\|_{\dot{H}^{-2\delta}}$.
	\end{Lemma}
	\begin{proof} 
		Applying the Gagliardo-Nirenberg inequality, we have
		\begin{align}\label{d80}
		\begin{aligned}
		\|\nabla u\|_{L^{1}([t_{0},T];L^\infty)}&\lesssim\|\nabla u\|^{\theta_1}_{L^{1}([t_{0},T];L^p)}\|\nabla^{2}u\|^{1-\theta_1}_{L^{1}([t_{0},T];L^r)},
		\end{aligned}
		\end{align}
		with $$0=\frac{\theta_1}{p}+(1-\theta_1)(\frac{1}{r}-\frac{1}{3})\quad{\rm or}\quad\theta_1=\frac{(r-3)p}{3r-3p+pr}.$$
		where $p\in[1,+\infty)$ is a positive constant which will be determined later and $r>3$. We now decompose the proof of Lemma \ref{Lemma-5.6} into the following steps:
		\vskip0.5cm
		\noindent\textbf{Step 1.} The estimate of $\|\nabla^{2}u\|_{L^{1}([t_{0},T];L^r)}.$
		
		By virtue of Lemma \ref{Lemma-5.1}, one has for $r\in(2,q]$ with $q\in(3,6),$
		\begin{align}\label{d76}
		\|\nabla^2 u\|_{L^r}\lesssim \|\nabla u\|_{L^2}+\|F\|_{L^r}+\|(-\Delta)^{-1}\operatorname{div}F\|_{L^2},
		\end{align}
		with $F=-\rho \partial_t u-\rho u\cdot\nabla u.$ Thus, for any $\eta>0$, we have
		$$\begin{aligned}
		\|F\|_{L^r}\lesssim &\|\rho u_t\|_{L^r}+\|u\|_{L^6}\|\nabla u\|_{L^{\frac{6r}{6-r}}}\\
		\lesssim& \|u_t\|^{\frac{6-r}{2r}}_{L^2}\|\nabla u_t\|^{\frac{3r-6}{2r}}_{L^2}+\|\nabla u\|^{\frac{6(r-1)}{5r-6}}_{L^2}\|\nabla^2 u\|^{\frac{4r-6}{5r-6}}_{L^r}\\
		\lesssim & \eta\|\nabla^2 u\|_{L^r}+\|u_t\|^{\frac{6-r}{2r}}_{L^2}\|\nabla u_t\|^{\frac{3r-6}{2r}}_{L^2}+\|\nabla u\|^{\frac{6(r-1)}{r}}_{L^2}.
		\end{aligned}$$
		Thanks to $\operatorname{div} u_t=0$, one has
		$$\begin{aligned}
		\|(-\Delta)^{-1}\operatorname{div}F\|_{L^2}\lesssim& \|(-\Delta)^{-1}\operatorname{div}((1-\rho)u_t)\|_{L^2}+\|\rho u\cdot\nabla u\|_{L^{\frac{6}{5}}}\\
		\lesssim & \|1-\rho_0\|_{L^2}\|u_t\|^{\frac{1}{2}}_{L^2}\|\nabla u_t\|^{\frac{1}{2}}_{L^2}+\|u\|^{\frac{1}{2}}_{L^2}\|\nabla u\|^{\frac{3}{2}}_{L^2},
		\end{aligned}$$
		where we have used $L^{\frac{6}{5}}\hookrightarrow \dot{W}^{-1,2}$. Substituting the above inequality into (\ref{d76}) and taking $\eta>0$ sufficiently small, we arrive at
		$$
		\begin{aligned}
		\|\nabla^2 u\|_{L^r}\lesssim \|\nabla u\|_{L^2}+\|u_t\|^{\frac{6-r}{2r}}_{L^2}\|\nabla u_t\|^{\frac{3r-6}{2r}}_{L^2}+\|\nabla u\|^{\frac{6(r-1)}{r}}_{L^2}+\|u_t\|^{\frac{1}{2}}_{L^2}\|\nabla u_t\|^{\frac{1}{2}}_{L^2}+\|u\|^{\frac{1}{2}}_{L^2}\|\nabla u\|^{\frac{3}{2}}_{L^2},
		\end{aligned}
		$$
		and it is easy to observe that
		\begin{align}\label{d78}
		\begin{aligned}
		&\int^{T}_{t_{0}}\|\nabla^{2} u\|_{L^{r}}dt\lesssim \int^{T}_{t_{0}}\|\nabla u\|_{L^2}dt+\int^{T}_{t_{0}}\|u_t\|^{\frac{6-r}{2r}}_{L^2}\|\nabla u_t\|^{\frac{3r-6}{2r}}_{L^2}dt\\&+\int^{T}_{t_{0}}\|\nabla u\|^{\frac{6(r-1)}{r}}_{L^2}dt+\int^{T}_{t_{0}}\|u_t\|^{\frac{1}{2}}_{L^2}\|\nabla u_t\|^{\frac{1}{2}}_{L^2}dt+\int^{T}_{t_{0}}\|u\|^{\frac{1}{2}}_{L^2}\|\nabla u\|^{\frac{3}{2}}_{L^2}dt\stackrel{\text { def }}{=}\sum^{5}_{i=1} I_{i}.
		\end{aligned}
		\end{align}
		First, due to Corollary \ref{Corollary-5.2}  and $\delta_{-}>\frac{1}{2}$, we get 
		$$
		\begin{aligned}
		I_1\leq \|t^{\delta_{-}}\nabla u\|_{L^{2}([t_{0},T];L^{2})}\bigg(\int^{T}_{t_{0}}t^{-2\delta_{-}}dt\bigg)^{\frac{1}{2}}\leq C\mathcal{H}_{0}.
		\end{aligned}
		$$
		Applying Lemma \ref{Lemma-5.5}, we arrive at
		$$
		\begin{aligned}
		I_2\leq &\int^{T}_{t_{0}}\|t\partial_{t}u\|^{\frac{6-r}{2r}}_{L^{2}}\|t\nabla u_{t}\|^{\frac{3(r-2)}{2r}}_{L^{2}}\cdot t^{-1}dt
		\leq \bigg(\sup_{t\in[t_{0},T]}\|t\partial_t u\|_{L^{2}}dt\bigg)^{\frac{6-r}{2r}}\\
		&\times\bigg(\int^{T}_{t_{0}}\|t\nabla u_{t}\|^{2}_{L^{2}}dt\bigg)^{\frac{3(r-2)}{4r}}\cdot\bigg(\int^{T}_{t_{0}}t^{-\frac{4r}{r+6}}dt\bigg)^{\frac{6+r}{4r}}
		\leq \exp\{C\mathcal{H}_{0}\}.
		\end{aligned}
		$$
		Similarly, 
		$$ 
		I_3\leq \sup_{t\in[t_{0},T]}\|\nabla u\|^{\frac{4r-6}{r}}_{L^{2}}\cdot\int^{T}_{t_{0}}\|\nabla u\|^{2}_{L^{2}}dt\leq C.
		$$
		The same estimate holds for $I_4$. Notice that
		$$
		\begin{aligned}
		I_5&\leq \sup_{t\in[t_0,T]}\|t^\delta u\|^{\frac{1}{2}}_{L^{2}}\bigg(\int^{T}_{t_{0}}\|t^{\delta_{-}}\nabla u\|^{2}_{L^{2}}dt\bigg)^{\frac{3}{4}} \bigg(\int^{T}_{t_{0}}t^{-8\delta_{-}}dt\bigg)^{\frac{1}{4}}\leq C\mathcal{H}_{0}.
		\end{aligned}
		$$
		Substituting the estimates of $I_1-I_5$ into the (\ref{d78}), we can deduce 
		\begin{align}\label{d79}
		\begin{aligned}
		\int^{T}_{t_{0}}\|\nabla^{2} u\|_{L^{r}}dt\leq \exp\{C\mathcal{H}_{0}\}.
		\end{aligned}
		\end{align}

		\noindent\textbf{Step 2.} The estimate of $\|\nabla u\|_{L^{1}([t_{0},T];L^p)}.$
		
		By virtue of the Gagliardo-Nirenberg inequality, 
		\begin{align}\label{dd58}
		\|t^{(1-\theta_2)\delta_{-}}\nabla u\|_{L^{2}([t_{0},T];L^{p})}\leq C\|\nabla u\|^{\theta_2}_{L^{2}([t_{0},T];L^{6})}\|t^{\delta_{-}}\nabla u\|^{1-\theta_2}_{L^{2}([t_{0},T];L^{2})}
		\end{align}
		with
		$$
		\begin{aligned}\frac{1}{p}=\frac{\theta_2}{6}+\frac{1-\theta_2}{2}\quad{\rm or}\quad \theta_2=\frac{3(p-2)}{2p}.
		\end{aligned}
		$$
		Thanks to Corollary \ref{Corollary-5.2}, we obtain 
		\begin{align}\label{dd59}
		\|t^{\delta_{-}}\nabla u\|^2_{L^{2}([t_{0},T];L^{2})}\leq C\mathcal{H}_{0},
		\end{align}
		Thanks to $u=v+w$ and Proposition \ref{Proposition-5.1}, one has
		$$
		\int^{T}_{t_{0}}\|\nabla v\|^{2}_{L^{6}}dt\leq \int^{T}_{t_{0}} \|v\|^2_{\dot{B}^2_{2,1}}dt\lesssim \|u_0\|^2_{\dot{B}^{\frac{1}{2}}_{2,1}},
		$$
		Combining Lemma \ref{Lemma-5.2} and Lemma \ref{Lemma-5.3}, we infer that
		$$
		\begin{aligned}
		\int^{T}_{t_{0}}&\|\nabla w\|^{2}_{L^{6}}dt \lesssim \int^{T}_{t_{0}}\|\nabla^2 w\|^{2}_{L^{2}}dt \lesssim\int^{T}_{t_{0}}\|\nabla w\|^{2}_{L^{2}}dt+\int^{T}_{t_{0}}\|\partial_{t}w\|^{2}_{L^{2}}dt\\&+\int^{T}_{t_{0}}\|\nabla^{2}w\|^{6}_{L^{2}}dt+\int^{T}_{t_{0}}\|v\|^2_{\dot{B}^{s_1}_{2,1}}dt
		\lesssim \int^{T}_{t_{0}}\|\nabla w\|^2_{L^{2}}dt+\int^{T}_{t_{0}}\|\partial_{t}w\|^{2}_{L^{2}}dt\\&+\sup_{t\in[t_{0},T]}\|\nabla w\|^{4}_{{L^{2}}}\int^{T}_{t_{0}}\|\nabla w\|^{2}_{L^{2}}dt+\int^{T}_{t_{0}}\|v\|^2_{\dot{B}^{s_1}_{2,1}}dt\lesssim \|u_0\|^2_{\dot{B}^{\frac{1}{2}}_{2,1}},
		\end{aligned}
		$$
		where we used the embedding inequality $\dot{H}^{1}\hookrightarrow L^{6}$ in the first inequality. Thus, we deduce that
		\begin{align}\label{dd60}
		\int^{T}_{t_{0}}\|\nabla u\|^{2}_{L^{6}}dt\leq \int^{T}_{t_{0}}\|\nabla v\|^{2}_{L^{6}}dt+\int^{T}_{t_{0}}\|\nabla w\|^{2}_{L^{6}}dt\lesssim \|u_0\|^2_{\dot{B}^{\frac{1}{2}}_{2,1}}.
		\end{align}
		\par Substituting (\ref{dd59}) and (\ref{dd60}) into (\ref{dd58}) and taking $p\in]2,\frac{6\delta_{-}}{1+\delta_{-}}[$ with $\delta\in]1/2,3/4[$, we get, by using H\"{o}lder's inequality, that
		$$
		\begin{aligned}
		\|\nabla u\|_{L^{1}(L^{p})}&\leq C	\|t^{\frac{6-p}{2p}\delta_{-}}\nabla u\|_{L^{2}([t_{0},T];L^{p})}\bigg(\int^{T}_{t_{0}}t^{\frac{p-6}{p}\delta_{-}}dt\bigg)^{\frac{1}{2}}\\&\leq \|u_0\|^{\frac{3(p-2)}{2p}}_{\dot{B}^{\frac{1}{2}}_{2,1}}\exp\{C\mathcal{H}_{0}\},
		\end{aligned}
		$$
		which along with (\ref{d79}) yields
		$$
		\begin{aligned}
		\|\nabla u\|_{L^{1}([t_{0},T];L^\infty)}&\leq\|\nabla u\|^{\theta_1}_{L^{1}([t_{0},T];L^{p})}\|\nabla^{2} u\|^{1-\theta_1}_{L^{1}([t_{0},T];L^{r})}\\&\leq \|u_{0}\|^{\frac{3\theta_1(p-2)}{2p}}_{\dot{B}^{\frac{1}{2}}_{2,1}}\exp\{C\mathcal{H}_{0}\}.
		\end{aligned}
		$$
		This completes the proof of Lemma \ref{Lemma-5.6}.
	\end{proof}
	\vskip 0.5cm
	\par With Lemma \ref{Lemma-u1}-\ref{Lemma-5.6} at hand, we are in a position to prove Proposition \ref{Proposition-5.5}.
	\vskip 0.5cm
	\noindent\textbf{Proof of Proposition \ref{Proposition-5.5}:}~Since $\mu(\rho)$ satisfies
	\begin{equation}\label{d83}
	\partial_{t}(\mu(\rho))+u\cdot\nabla \mu(\rho)=0,
	\end{equation}
	standard caclulations show that
	\begin{equation}\label{d84}
	\frac{d}{dt}\|\nabla\mu(\rho)\|_{L^{q}}\leq q \|\nabla u\|_{L^{\infty}}\|\nabla\mu(\rho)\|_{L^{q}},
	\end{equation}
	which together with Gronwall's inequality and (\ref{d75}) yields
	\begin{equation}\label{d85}
	\begin{aligned}
	\sup_{t\in[0, T]}\|\nabla\mu(\rho)\|_{L^{q}}
	\leq&\|\nabla \mu(\rho_0)\|_{L^{q}}\exp{\{q\int^{t_0}_{0}\|\nabla u\|_{L^{\infty}}dt+q\int^{T}_{t_0}\|\nabla u\|_{L^{\infty}}dt\}}\\
	\leq& \|\nabla\mu(\rho_0)\|_{L^{q}}\exp\bigg\{C\|u_0\|_{\dot{B}^{\frac{1}{2}}_{2,1}}+\|u_{0}\|^{\frac{3\theta_1(p-2)}{2p}}_{\dot{B}^{\frac{1}{2}}_{2,1}}\exp\{C\mathcal{H}_{0}\}\bigg\}.
	\end{aligned}
	\end{equation}
	Hence, one has
	$$\sup_{t\in[t_{0}, T]}\|\nabla\mu(\rho)\|_{L^{q}}\leq 2\|\nabla\mu(\rho_0)\|_{L^{q}},$$
	provided that
	\begin{equation}\label{d86}
	\|u_0\|_{\dot{B}^{\frac{1}{2}}_{2,1}}\leq\varepsilon_{2}\quad{\rm and}\quad\varepsilon_{2}+\varepsilon^{\frac{3\theta_1(p-2)}{2p}}_{2}\exp\{\mathcal{H}_0\}\overset{\text{def}}{=}C^{-1}\ln 2.
	\end{equation}
	\par Choosing $\varepsilon=\min\{1,\varepsilon_{1},\varepsilon_{2}\}$, we directly obtain (\ref{d40}) from (\ref{d84})-(\ref{d86}). The proof of Proposition \ref{Proposition-5.5} is finished.
	
	\subsection{Proof of Theorem \ref{Theorem-1}}
	
	\par  We first rewrite the momentum equation in $(\ref{a1})_{2}$ as 
	\begin{align}\label{d87}
	\partial_{t}u+u\cdot\nabla u-\mu(\rho)\Delta u+\nabla\pi=(1-\rho)(\partial_t u+u\cdot\nabla u)+2\nabla\mu(\rho)\mathbb{D}u.
	\end{align}
	Applying the operator $\dot{\Delta}_{j}\mathbb{P}$ to (\ref{d87}) gives
	\begin{align}\label{d88}
	\begin{aligned}
	\partial_{t}\dot{\Delta}_{j}u+u\cdot\nabla \dot{\Delta}_{j}u-&\operatorname{div}\{\mu(\rho)\dot{\Delta}_{j}\nabla u\}=\dot{\Delta}_j\mathbb{P}\{(1-\rho)(\partial_t u+u\cdot\nabla u)\} +[u\cdot\nabla;\dot{\Delta}_{j}\mathbb{P}]u\\&-[\mu(\rho);\dot{\Delta}_{j}\mathbb{P}]\Delta u-\nabla\mu(\rho)\cdot\dot{\Delta}_{j}\nabla u+\dot{\Delta}_{j}\mathbb{P}(2\nabla\mu(\rho)\mathbb{D}u).
	\end{aligned}
	\end{align}
	Due to $\operatorname{div}u=0$ and $\mu(\rho)\geq \underline{\mu}$ by multiplying (\ref{d88}) by $\dot{\Delta}_{j}u$ and then integrating the resulting equality over $\mathbb{R}^{3}$, we obtain
	\begin{align}\label{d89}
	\begin{aligned}
	\frac{d}{dt}\|\dot{\Delta}_{j} u\|_{L^{2}}+&2^{2j}\|\dot{\Delta}_{j} u\|_{L^{2}}\lesssim \|\dot{\Delta}_j\{(1-\rho)(\partial_t u+u\cdot\nabla u)\}\|_{L^2}+\|[u\cdot\nabla;\dot{\Delta}_{j}\mathbb{P}]u\|_{L^{2}}\\&+\|[\mu(\rho);\dot{\Delta}_{j}\mathbb{P}]\Delta u\|_{L^{2}}+\|\nabla\mu(\rho)\cdot\dot{\Delta}_{j}\nabla u\|_{L^2}+\|\dot{\Delta}_{j}\mathbb{P}(2\nabla\mu(\rho)\mathbb{D}u)\|_{L^2}.
	\end{aligned}
	\end{align} 
	Intergrating the above inequality over $[t_0,T]$ and multiplying (\ref{d89}) by $2^{j(\frac{1}{2})}$, then summing up the resulting inequality over $j\in\mathbb{Z},$ we achieve 
	\begin{align}\label{d90}
	\begin{aligned}
	&\|u\|_{\widetilde{L}^{\infty}([t_{0},T];\dot{B}^{\frac{1}{2}}_{2,1})}+\|u\|_{L^{1}([t_{0},T];\dot{B}^{\frac{5}{2}}_{2,1})}
	\\\lesssim& \|u(t_{0})\|_{\dot{B}^{\frac{1}{2}}_{2,1}}+\|(1-\rho)(\partial_t u+u\cdot\nabla u)\|_{L^{1}([t_{0},T];\dot{B}^{\frac{1}{2}}_{2,1})}\\&+ \sum_{j\in\mathbb{Z}}2^{\frac{j}{2}}\|[u\cdot\nabla;\dot{\Delta}_{j}\mathbb{P}]u\|_{L^{1}([t_{0},T];L^{2})}+\sum_{j\in\mathbb{Z}}2^{\frac{j}{2}}\|[\mu(\rho);\dot{\Delta}_{j}\mathbb{P}]\Delta u\|_{L^{1}([t_{0},T];L^{2})}\\
	&+\sum_{j\in\mathbb{Z}}2^{\frac{j}{2}}\|\nabla\mu(\rho)\cdot\dot{\Delta}_{j}\nabla u\|_{L^{1}([t_{0},T];L^{2})}+\|\nabla\mu(\rho)\mathbb{D}u\|_{L^{1}([t_{0},T];\dot{B}^{\frac{1}{2}}_{2,1})}.
	\end{aligned}
	\end{align} 
	\par In what follows,~we shall deal with the right-hand side of~(\ref{d90}). Applying Lemma \ref{Lemma-2.2} gives
	$$
	\begin{aligned}
	\|(1-\rho)(\partial_t u+u\cdot\nabla u)\|_{L^{1}([t_{0},T];\dot{B}^{\frac{1}{2}}_{2,1})}
	\lesssim \|1-\rho\|_{L^\infty([t_0,T];\dot{B}^{\frac{1}{2}}_{6,1})}\|\partial_t u+u\cdot\nabla u\|_{L^{1}([t_{0},T];\dot{B}^{\frac{1}{2}}_{2,1})}.
	\end{aligned}$$
	Yet thanks to Lemma \ref{Lemma-5.5} and (\ref{d3}), one has 
	$$
	\begin{aligned}
	\|u_t\|_{L^{1}([t_{0},T];\dot{B}^{\frac{1}{2}}_{2,1})}\leq  \sup_{t\in[t_0,T]}\|t u_t\|^{\frac{1}{2}}_{L^2}\bigg(\int^{T}_{t_0}\|t\nabla u_t\|^{2}_{L^2}dt\bigg)^{\frac{1}{4}}\bigg(\int^{T}_{t_0}t^{-\frac{4}{3}}dt\bigg)^{\frac{3}{4}}\leq \exp\{C\mathcal{H}_{0}\},
	\end{aligned}
	$$
	and
	$$
	\begin{aligned}
	\|u\cdot\nabla u\|_{L^{1}([t_{0},T];\dot{B}^{\frac{1}{2}}_{2,1})}\lesssim \int^{T}_{t_0}(\|\nabla w\|_{L^2}\|\Delta w\|_{L^2}+\|v\|^{2}_{\dot{B}^{\frac{3}{2}}_{2,1}})dt\leq \|u_0\|^2_{\dot{B}^{\frac{1}{2}}_{2,1}}.
	\end{aligned}
	$$
	Whereas thanks to Thenrem 2.87 in \cite{2011BCD}, we have
	$$
	\begin{aligned}
	\|1-\rho\|_{\widetilde{L}^\infty([t_0,T];\dot{B}^{\frac{1}{2}}_{6,1})}&\lesssim \|1-\rho(t_0)\|_{\dot{B}^{\frac{1}{2}}_{6,1}}\exp\bigg\{\int^{T}_{t_0}\|\nabla u\|_{L^{\infty}}dt\bigg\}.
	\end{aligned}
	$$
	Yet thanks to Lemma \ref{Lemma-5.6} and (\ref{d3}), one has
	\begin{align}\label{d91}
	\int^{T}_{t_{0}}\|\nabla u\|_{L^{\infty}}dt\leq C\|u_0\|_{\dot{B}^{\frac{1}{2}}_{2,1}}+\|u_{0}\|^{\frac{3\theta_1(p-2)}{2p}}_{\dot{B}^{\frac{1}{2}}_{2,1}}\exp\{C\mathcal{H}_{0}\}.
	\end{align}
	From which, we can deduce that
	\begin{align}\label{d92}
	\begin{aligned}
	&\|(1-\rho)(\partial_t u+u\cdot\nabla u)\|_{L^{1}([t_{0},T];\dot{B}^{\frac{1}{2}}_{2,1})}\\
	\lesssim& \|1-\rho(t_0)\|_{\dot{B}^{\frac{1}{2}}_{6,1}}\exp\bigg\{C\|u_0\|_{\dot{B}^{\frac{1}{2}}_{2,1}}+\|u_{0}\|^{\frac{3\theta_1(p-2)}{2p}}_{\dot{B}^{\frac{1}{2}}_{2,1}}\exp\{C\mathcal{H}_{0}\}\bigg\}.
	\end{aligned}
	\end{align}
	Similarly, by virtue of Lemma \ref{Lemma-2.3}, 
	\begin{align}\label{d93}
	\sum_{j\in\mathbb{Z}}2^{\frac{j}{2}}\|[u\cdot\nabla;\dot{\Delta}_{j}\mathbb{P}]u\|_{L^{1}([t_{0},T];L^{2})}\lesssim\int^{T}_{t_{0}} \|\nabla u\|_{L^{\infty}}\|u\|_{\dot{B}^{\frac{1}{2}}_{2,1}}dt.
	\end{align}
	For $[\mu(\rho);\dot{\Delta}_{j}\mathbb{P}]\Delta u,$ homogenous Bony's decomposition implies
	$$
	[\mu(\rho);\dot{\Delta}_{j}\mathbb{P}]\Delta u=[T_{\mu(\rho)};\dot{\Delta}_{j}\mathbb{P}]\Delta u+T^{'}_{\dot{\Delta}_{j}\Delta u}\mu(\rho)-\dot{\Delta}_{j}\mathbb{P}(T^{'}_{\Delta u}\mu(\rho)).$$
	It follows again from the above estimate,~which implies that
	$$
	\begin{aligned}
	\sum_{j\in\mathbb{Z}}2^{\frac{j}{2}}\|[T_{\mu(\rho)};\dot{\Delta}_{j}\mathbb{P}]\Delta u\|_{L^{2}}
	\lesssim&\sum_{j\in\mathbb{Z}}\sum_{|j-k|\leq 4}2^{\frac{1}{2}(k-j)}\|\nabla\dot{S}_{k-1}\mu(\rho)\|_{L^{q}}2^{-\frac{k}{2}}\|\dot{\Delta}_{k}\Delta u\|_{L^{q^{*}}}\\
	\lesssim&\sum_{j\in\mathbb{Z}}\sum_{|j-k|\leq 4}2^{\frac{1}{2}(k-j)}\|\nabla \mu(\rho)\|_{L^{q}}2^{(\frac{3}{q}-\frac{1}{2})k}\|\dot{\Delta}_{k}\nabla u\|_{L^{2}}\\
	\lesssim& \|\nabla \mu(\rho)\|_{L^{q}}\|u\|_{\dot {B}^{\frac{3}{q}+\frac{3}{2}}_{2,1}},
	\end{aligned}
	$$
	where $\frac{1}{q}+\frac{1}{q^{*}}=\frac{1}{2}.$ Similarly, one has
	$$
	\begin{aligned}
	\sum_{j\in\mathbb{Z}}2^{\frac{j}{2}}\|T^{'}_{\dot{\Delta}_{j}\Delta u} \mu(\rho)\|_{L^{2}}&\lesssim\sum_{j\in\mathbb{Z}}\sum_{k\geq j-2}2^{\frac{j}{2}}\|\dot{S}_{k+2}\dot{\Delta}_{j}\Delta u\|_{L^{q^{*}}}\|\dot{\Delta}_{k} \mu(\rho)\|_{L^{q}}\\&
	\lesssim\sum_{j\in\mathbb{Z}}2^{-\frac{j}{2}}\|\dot{\Delta}_{j}\Delta u\|_{L^{q^{*}}}\sum_{k\geq j-2}2^{j-k}\|\dot{\Delta}_{k}\nabla \mu(\rho)\|_{L^{q}}\\&
	\lesssim \|u\|_{\dot {B}^{\frac{3}{q}+\frac{3}{2}}_{2,1}}\|\nabla \mu(\rho)\|_{L^{q}}.
	\end{aligned}
	$$
	Notice that
	$$
	\begin{aligned}
	&\sum_{j\in\mathbb{Z}}2^{\frac{j}{2}}\|\dot{\Delta}_{j}\mathbb{P}(T_{\Delta u} \mu(\rho))\|_{L^{2}}\\\lesssim&\sum_{j\in\mathbb{Z}}\sum_{|k-j|\leq 4}2^{\frac{1}{2}(j-k)}2^{-\frac{k}{2}}\|\dot{S}_{k-1}\Delta u\|_{L^{q^{*}}}2^{k}\|\dot{\Delta}_{k}\mu(\rho)\|_{L^{q}}\\
	\lesssim &\|u\|_{\dot {B}^{\frac{3}{q}+\frac{3}{2}}_{2,1}}\|\nabla \mu(\rho)\|_{L^{q}}.
	\end{aligned}
	$$
	The same estimate hold for $\dot{\Delta}_{j}\mathbb{P}\mathcal{R}(\Delta u,\mu(\rho)).$ From which, we obtain
	\begin{equation}\label{d94}
	\begin{aligned}
	&\sum_{j\in\mathbb{Z}}2^{\frac{j}{2}}\|[\mu(\rho);\dot{\Delta}_{j}\mathbb{P}]\Delta u\|_{L^{1}([t_{0},T];L^{2})} \lesssim\int^{T}_{t_{0}}\|u\|_{\dot {B}^{\frac{3}{q}+\frac{3}{2}}_{2,1}}\|\nabla \mu(\rho)\|_{L^{q}}dt.
	\end{aligned}
	\end{equation}
	Along the same line, one has
	\begin{equation}\label{d95}
	\begin{aligned}
	&\sum_{j\in\mathbb{Z}}2^{\frac{j}{2}}\|\nabla\mu(\rho)\cdot\dot{\Delta}_{j}\nabla u\|_{L^{1}([t_{0},T];L^{2})}+\|\nabla\mu(\rho)\mathbb{D}u\|_{L^{1}([t_{0},T];\dot{B}^{\frac{1}{2}}_{2,1})}\\
	\lesssim&\int^{T}_{t_{0}}\|u\|_{\dot {B}^{\frac{3}{q}+\frac{3}{2}}_{2,1}}\|\nabla \mu(\rho)\|_{L^{q}}dt+\int^{T}_{t_{0}}\|\nabla u\|_{L^\infty}\|1-\rho\|_{\dot {B}^{\frac{3}{2}}_{2,1}}dt.
	\end{aligned}
	\end{equation}
	By virtue of Lemma \ref{Lemma-2.4} and (\ref{d79}) with $r=3$, one has
	\begin{align}\label{d96}
	\|1-\rho\|_{\widetilde{L}_{T}^{\infty}(\dot{B}^{\frac{3}{2}}_{2,1})}\leq\|1-\rho(t_0)\|_{\dot{B}^{\frac{3}{2}}_{2,1}}\exp\bigg\{C\int^{T}_{t_0}\|\nabla u\|_{\dot{B}^{1}_{3,\infty}\cap L^\infty}dt\bigg\}.
	\end{align}
	\par Plugging the above estimates into (\ref{d90}) and applying interpolation inequality $\|u\|_{\dot {B}^{\frac{3}{q}+\frac{3}{2}}_{2,1}}\lesssim \|u\|^{\frac{q-3}{2q}}_{\dot{B}^{\frac{1}{2}}_{2,1}}\|u\|^{\frac{q+3}{2q}}_{\dot{B}^{\frac{5}{2}}_{2,1}}$ and Young's inequality, we arrive at
	\begin{align}
	\begin{aligned}
	\|u&\|_{\widetilde{L}^{\infty}([t_{0},T];\dot{B}^{\frac{1}{2}}_{2,1})}+\|u\|_{L^{1}([t_{0},T];\dot{B}^{\frac{5}{2}}_{2,1})}\leq\|u(t_0)\|_{\dot{B}^{\frac{1}{2}}_{2,1}}+\int^{T}_{t_{0}} \|\nabla u\|_{L^{\infty}}\|u\|_{\dot{B}^{\frac{1}{2}}_{2,1}}dt\\
	&\quad+\int^{T}_{t_{0}}\|u\|_{\dot {B}^{\frac{1}{2}}_{2,1}}\|\nabla \mu(\rho)\|^{\frac{2q}{q-3}}_{L^{q}}dt+\|1-\rho(t_0)\|_{\dot {B}^{\frac{3}{2}}_{2,1}}\exp\{\exp\{C\mathcal{H}_{0}\}\},
	\end{aligned}
	\end{align}
	which together with Gronwall's inequality yields 
	\begin{align}\label{d97}
	\begin{aligned}
	\|u\|_{\widetilde{L}^{\infty}([t_{0},T];\dot{B}^{\frac{1}{2}}_{2,1})}+&\|u\|_{L^{1}([t_{0},T];\dot{B}^{\frac{5}{2}}_{2,1})}
	\leq\bigg(\|u(t_0)\|_{\dot{B}^{\frac{1}{2}}_{2,1}}+\|1-\rho(t_0)\|_{\dot{B}^{\frac{3}{2}}_{2,1}}\bigg)\\&\times\exp\bigg\{\{1+T\|\nabla \mu(\rho_0)\|^{\frac{2q}{q-3}}_{L^{q}}\}\exp\{\exp\{C\mathcal{H}_{0}\}\}\bigg\}.
	\end{aligned}
	\end{align} 
	The similar estimates hold for $\nabla\pi$ and $u_t.$ Hence, we completes the proof of Theorem \ref{Theorem-1}.

	\begin{appendix}
		\section{The large time decay-estimate of $u$}
		
		\par In order to get the optimal decay of $u$, we need the following Lemma:
		
		\begin{Lemma}\label{Lemma-A.1}
			Suppose $(\rho,u,\pi)$ is the unique local strong solution to $(\ref{a1})$ satisfying for any $t\in[t_0,T],$
			\begin{equation}\label{e1}
			\|u\|^2_{L^{2}}\leq C\mathcal{H}_{0}\langle t\rangle^{-\frac{1}{2}}\quad {\rm with} \quad \langle t\rangle\stackrel{\mathrm{ def }}{=}1+t,
			\end{equation} 
			then under the assumptions of $(\ref{e1})$ and Theorem \ref{Theorem-1}, we have 
			\begin{equation}\label{e2}
			\sup_{t\in[t_{0},T]}\|\langle t\rangle^{\frac{1}{4}}u\|^2_{L^2}+	\int^{T}_{t_{0}}\|\langle t\rangle^{\frac{1}{4}_{-}}\nabla u\|^2_{L^2}dt\leq C\mathcal{H}_{0},
			\end{equation}
			and
			\begin{equation}\label{e3}
			\begin{aligned}
			\int^{T}_{t_{0}}\|\langle t\rangle^{\frac{1}{2}_{-}}u_{t}\|^{2}_{L^{2}}dt+\sup_{t\in[t_{0},T]}\|\langle t\rangle^{\frac{1}{2}_{-}}\nabla u\|^{2}_{L^{2}}\leq C\mathcal{H}_{0},
			\end{aligned}
			\end{equation}
			where 
			\begin{equation}\label{h}
			\mathcal{H}_{0}\stackrel{\mathrm{ def }}{=}1+\|u(t_{0})\|^{2}_{\dot{H}^{-2\delta}}+\|u(t_{0})\|^{2}_{H^1}(1+\|1-\rho_{0}\|^{2}_{L^{2}}+\|u(t_{0})\|^{2}_{L^{2}}).
			\end{equation}
		\end{Lemma}
		\begin{proof}
			We first get by using standard energy estimate to the $u$ equation of (\ref{a1}) that
			\begin{equation}\label{e4}
			\frac{d}{dt}\|\sqrt{\rho}u\|^2_{L^2}+2\underline{\mu}\|\nabla u\|^2_{L^2}\leq 0.
			\end{equation}
			Multiplying (\ref{e4}) by $\langle t\rangle^{\frac{1}{2}_{-}},$ one has
			\begin{equation}\label{ee2}
			\begin{aligned}
			&\|\langle t\rangle^{\frac{1}{4}_{-}}u\|^2_{L^\infty([t_0,T];L^2)}+\|\langle t\rangle^{\frac{1}{4}_{-}}\nabla u\|^2_{L^2([t_0,T];L^2)}\\
			\lesssim& \|u(t_0)\|^2_{L^2}+\int^{T}_{t_{0}}\langle t\rangle^{(-1)_{-}}\|\langle t\rangle^{\frac{1}{4}}u(t)\|^{2}_{L^{2}}dt\lesssim\mathcal{H}_{0}.
			\end{aligned}
			\end{equation}
			Multiplying (\ref{dd53}) by $\langle t\rangle^{1_{-}}$ and then integrating the resulting inequality over $[t_{0},T]$, we get, by applying (\ref{ee2}), that
			\begin{equation}
			\begin{aligned}
			&\int^{T}_{t_{0}}\|\langle t\rangle^{\frac{1}{2}_{-}}u_{t}\|^{2}_{L^{2}}dt+\sup_{t\in[t_{0},T]}\|\langle t\rangle^{\frac{1}{2}_{-}}\nabla u\|^{2}_{L^{2}}
			\lesssim  \|\nabla u(t_0)\|^2_{L^2}\\&
			\quad+\int^{T}_{t_{0}}\langle t\rangle^{(1_{-})-1}\|\nabla u\|^{2}_{L^{2}}dt+\int^{T}_{t_{0}}\|\langle t\rangle^{\frac{1}{2}_{-}}\nabla u\|_{L^{2}}\|\langle t\rangle^{\frac{1}{4}_{-}}\nabla u\|^{2}_{L^{2}}dt\\&\quad+\int^{T}_{t_{0}}\|\langle t\rangle^{\frac{1}{2}_{-}}\nabla u\|^2_{L^{2}}\|\nabla u\|^{2}_{L^{2}}dt+\int^{T}_{t_{0}}\|\langle t\rangle^{\frac{1}{2}_{-}}\nabla u\|^2_{L^{2}}\|\nabla u\|^{4}_{L^{2}}dt.
			\end{aligned}
			\end{equation}
			Hence, one has
			$$\int^{T}_{t_{0}}\|\langle t\rangle^{\frac{1}{2}_{-}}u_{t}\|^{2}_{L^{2}}dt+\sup_{t\in[t_{0},T]}\|\langle t\rangle^{\frac{1}{2}_{-}}\nabla u\|^{2}_{L^{2}}\leq C\mathcal{H}_{0},$$
		\end{proof}
		
		\begin{Proposition}\label{Proposition-A.1}
			Under the assumption of Theorem \ref{Theorem-1}, we have
			\begin{equation}\label{e13}
			\|u\|^{2}_{L^{2}}\leq C\mathcal{H}_{0}\langle t\rangle ^{-2\delta} \quad{\rm for~any}~ t\in[t_{0},T],
			\end{equation}
			where $\mathcal{H}_{0}$ given by (\ref{h}).
		\end{Proposition}
		\begin{proof} 
			Taking $L^{2}$ inner product of $(\ref{a1})_{1}$ with $u,$ and using the fact $\operatorname{div}u=0,$ we obtain
			\begin{equation}\label{e14}
			\frac{d}{dt}\|\sqrt{\rho}u\|^2_{L^2}+2\underline{\mu}\|\nabla u\|^{2}_{L^{2}}\leq 0.
			\end{equation}
			Applying Schonbek's strategy \cite{1986S}, by splitting the space~$\mathbb{R}^{3}$~into the time-dependent domain
			$$\mathbb{R}^{2}=S(t)\cup S^{c}(t),$$
			where $S(t)\stackrel{\text { def }}{=}\{{\xi\in\mathbb{R}^{2}:|\xi|\leq \sqrt{\frac{\bar{\rho}}{2\underline{\mu}}} r(t)}\}$ for some $r(t),$ will be chosen later on.
			\par Then we deduce from~$(\ref{a1})_{2}$~that
			\begin{equation}\label{e15}
			\frac{d}{dt}\|\sqrt{\rho}u\|^{2}_{L^{2}}+r^{2}(t)\|\sqrt{\rho}u\|^{2}_{L^{2}}\leq C r^{2}(t)\int_{S(t)}|\hat{u}(t,\xi)|^{2}d\xi.
			\end{equation}
			By the~Duhamel~principle, we can write $(\ref{a1})_{2}$ as the following integral form:
			$$\begin{aligned}
			u(t)=e^{(t-t_{0})\Delta}u(t_{0})+&\int^{t}_{t_{0}}e^{(t-s)\Delta}\mathbb{P}(\operatorname{div}(2(\mu(\rho)-1)\mathbb{D}u)\\&+(1-\rho)\partial_t u-\rho (u\cdot\nabla )u)ds.
			\end{aligned}$$
			Taking the Fourier transform with respect to~$x$~variables to the above equation, one has
			\begin{equation}\label{e16}
			\begin{aligned}
			|\hat u(t,\xi)|\leq& e^{-(t-t_{0})|\xi|^{2}}|\hat{u}(t_{0},\xi)|+\int^{t}_{t_{0}}e^{-(t-s)|\xi|^{2}}|\xi||\mathcal{F}_{x}(2(\mu(\rho)-1)\mathbb{D}u\\&-u\otimes u)|ds+\int^{t}_{t_{0}}e^{-(t-s)|\xi|^{2}}|\mathcal{F}_{x}((1-\rho)(\partial_t u+ u\cdot\nabla u))|ds,
			\end{aligned}
			\end{equation}
			so that
			\begin{equation}\label{e17}
			\begin{aligned}
			\int_{S(t)} |\hat u(t,\xi)|^{2}d\xi
			\lesssim& \int_{S(t)}e^{-2(t-t_{0})|\xi|^{2}}|\hat{u}(t_{0},\xi)|^{2}d\xi+\int_{S(t)}|\xi|^{2}(\int^{t}_{t_{0}}|\mathcal{F}_{x}((\mu(\rho)-1)\mathbb{D}u)|ds)^{2}d\xi\\&+\int_{S(t)}|\xi|^{2}(\int^{t}_{t_{0}}|\mathcal{F}_{x}(u\otimes u)|ds)^{2}d\xi+\int_{S(t)}(\int^{t}_{t_{0}}|\mathcal{F}_{x}((1-\rho)(\partial_t u+ u\cdot\nabla u))|ds)^{2}d\xi.
			\end{aligned}
			\end{equation}
			We now estimate term by term above. It is easy to observe that
			$$
			\int_{S(t)}e^{-2(t-t_{0})|\xi|^{2}}|\hat{u}(t_{0},\xi)|^{2}d\xi\lesssim \langle t\rangle^{-2\delta }\|u(t_{0})\|^{2}_{L^{2}\cap\dot{H}^{-2\delta }},
			$$
			and
			$$
			\begin{aligned}
			&\int_{S(t)}|\xi|^{2}(\int^{t}_{t_{0}}|\mathcal{F}_{x}((\mu(\rho)-1)\mathbb{D}u)|ds)^{2}d\xi\\
			\lesssim& r^{5}(t)(\int^{t}_{t_{0}}\|(\mu(\rho)-1)\mathbb{D}u\|_{L^{1}}dt)^{2}\lesssim r^{5}(t)(\int^{t}_{t_{0}}\|\mu(\rho)-1\|_{L^{2}}\|\nabla u\|_{L^{2}}dt)^{2}.
			\end{aligned}
			$$
			Similarly,
			$$
			\begin{aligned}
			\int_{S(t)}|\xi|^{2}(\int^{t}_{t_{0}}|\mathcal{F}_{x}(u\otimes u)|ds)^{2}d\xi&\lesssim r^{5}(t)(\int^{t}_{t_{0}}\|u\otimes u\|_{L^{1}}ds)^{2}\\&\lesssim r^{5}(t)(\int^{t}_{t_{0}}\|u\|^{2}_{L^{2}}ds)^{2}.
			\end{aligned}
			$$
			Whereas by virtue of (\ref{u}), we write
			$$(\int^t_{t_0}\|\partial_t u\|_{L^2}ds)^2\leq \langle t\rangle\int^{t}_{t_0}\|\partial_t u\|^2_{L^2}ds\lesssim \langle t\rangle,$$
			and
			$$(\int^{t}_{t_0}\|u\|_{L^2}\|\nabla u\|_{L^2}ds)^2\leq \int^{t}_{t_0}\|u\|^2_{L^2}ds\cdot\int^{t}_{t_0}\|\nabla u\|^2_{L^2}ds\lesssim \langle t\rangle \|u(t_0)\|^4_{L^2}, $$
			which gives
			$$
			\begin{aligned}
			&\int_{S(t)}(\int^{t}_{t_{0}}|\mathcal{F}_{x}((1-\rho)(\partial_t u+ u\cdot\nabla u))|ds)^{2}d\xi\lesssim r^3(t)\langle t\rangle(\|1-\rho_0\|^2_{L^2}\|\nabla u(t_0)\|^2_{L^2}+\|u(t_0)\|^4_{L^2}).
			\end{aligned}
			$$
			Resunming the above esitimates into (\ref{e17}) yields
			$$
			\begin{aligned}
			\int_{S(t)} |\hat u(t,\xi)|^{2}d\xi
			\lesssim &\langle t\rangle^{-2\delta }\|u(t_{0})\|^{2}_{L^{2}\cap\dot{H}^{-2\delta }}+r^{5}(t)\langle t\rangle\|\mu(\rho)-1\|^{2}_{L^{\infty}_{t}(L^{2})}\|u(t_{0})\|^{2}_{L^{2}}\\&+ r^{5}(t)\langle t\rangle^{2}\|u(t_{0})\|^{4}_{L^{2}}+r^3(t)\langle t\rangle(\|\mu(\rho_0)-1\|^2_{L^2}\|\nabla u(t_0)\|^2_{L^2}+\|u(t_0)\|^4_{L^2}).
			\end{aligned}
			$$
			Inserting the resulting inequality into (\ref{e15}), we get
			\begin{equation}\label{e18}
			\begin{aligned}
			&\frac{d}{dt}\|\sqrt{\rho}u\|^{2}_{L^{2}}+r^{2}(t)\|\sqrt{\rho}u\|^{2}_{L^{2}}
			\lesssim \mathcal{H}_{0}\{r^{2}(t)\langle t\rangle^{-2\delta }+r^{7}(t)\langle t\rangle+r^{7}(t)\langle t\rangle^{2}+r^{5}(t)\langle t\rangle\},
			\end{aligned}
			\end{equation}
			where $$\mathcal{H}_{0}\stackrel{\text { def }}{=}1+\|u(t_{0})\|^{2}_{\dot{H}^{-2\delta}}+\|u(t_{0})\|^{2}_{H^1}(1+\|1-\rho_0\|^{2}_{L^{2}}+\|u(t_{0})\|^{2}_{L^{2}}).$$
			From which, we can deduce that
			\begin{equation}\label{e19}
			\begin{aligned}
			\frac{d}{dt}\{e^{\int^{t}_{t_{0}}r^{2}(t)ds}\|\sqrt{\rho}u\|^{2}_{L^{2}}\}
			\lesssim \mathcal{H}_{0}e^{\int^{t}_{t_{0}}r^{2}(t)ds}\{r^{2}(t)\langle t\rangle^{-2\delta }+r^{7}(t)\langle t\rangle+r^{7}(t)\langle t\rangle^{2}+r^{5}(t)\langle t\rangle\}.
			\end{aligned}
			\end{equation}
			\par Taking $r^{2}(t)=\beta \langle t\rangle^{-1},$~(with $\beta>2\delta$) in the case when $\delta\in(0,\frac{1}{4}]$, we deduce from (\ref{e19}) that
			$$
			\begin{aligned}
			\frac{d}{dt}\{\langle t\rangle^{\beta}\|\sqrt{\rho}u\|^{2}_{L^{2}}\}&\lesssim \mathcal{H}_{0}\langle t\rangle ^{\beta}\{\langle t\rangle^{-1-2\delta}+\langle t\rangle^{-\frac{5}{2}}+\langle t\rangle^{-\frac{3}{2}}\}\\
			&\lesssim \mathcal{H}_{0}\langle t\rangle ^{-1-2\delta+\beta},
			\end{aligned}
			$$
			and integrating the results over~$[t_{0},t]$~one has
			\begin{equation}\label{e20}
			\|u\|^{2}_{L^{2}}\lesssim \mathcal{H}_{0}\langle t\rangle ^{-2\delta}.
			\end{equation}
			\par In the case  when $\delta\in(\frac{1}{4},\frac{3}{4}),$ we already have
			$$
			\|u\|^{2}_{L^{2}}\lesssim \mathcal{H}_{0}\langle t\rangle^{-\frac{1}{2}},
			$$
			from which, we infer
			$$
			\begin{aligned}
			(\int^{t}_{t_{0}}\|\mathcal{F}_{x}(u\otimes u)\|_{L^{\infty}_{\xi}}ds)^{2}&\lesssim (\int^{t}_{t_{0}}\|u\otimes u\|_{L^{1}}ds)^{2}\lesssim (\int^{t}_{t_{0}}\|u\|^{2}_{L^{2}}ds)^{2}\lesssim\mathcal{H}_{0}\langle t\rangle.
			\end{aligned}
			$$
			By virtue of Lemma \ref{Lemma-A.1}, we infer
			$$(\int^t_{t_0}\|\partial_t u\|_{L^2}ds)^2\leq \langle t\rangle^{1-(1_{-})}\int^{t}_{t_0}\|\langle s\rangle^{\frac{1}{2}_{-}}\partial_{t}u(s)\|^2_{L^2}ds\lesssim\mathcal{H}_{0} \langle t\rangle^{1-(1_{-})},$$
			and
			$$
			\begin{aligned}
			(\int^{t}_{t_0}\|u\|_{L^2}\|\nabla u\|_{L^2}ds)^2
			\leq& \mathcal{H}_{0}\int^{t}_{t_0}\langle s\rangle^{-(1_{-})}ds\cdot\sup_{t\in[t_0,t]}\langle t\rangle^{\frac{1}{2}}\|u\|^2_{L^2}\cdot\int^{t}_{t_0}\|\langle s\rangle^{\frac{1}{4}_{-}}\nabla u\|^2_{L^2}ds\\
			\lesssim& \mathcal{H}_{0}\langle t\rangle^{1-(1_{-})} ,
			\end{aligned}
			$$
			which gives
			$$
			\begin{aligned}
			\int_{S(t)}(\int^{t}_{t_{0}}|\mathcal{F}_{x}((1-\rho)(\partial_t u+ u\cdot\nabla u))|ds)^{2}d\xi\lesssim \mathcal{H}_{0}r^3(t)\langle t\rangle^{1-(1_{-})}.
			\end{aligned}
			$$
			Hence a similar derivation of (\ref{e18}) leads to
			\begin{equation}\label{e21}
			\begin{aligned}
			&\frac{d}{dt}\|\sqrt{\rho}u\|^{2}_{L^{2}}+r^{2}(t)\|\sqrt{\rho}u\|^{2}_{L^{2}}\\
			\leq &C\mathcal{H}_{0}\{r^{2}(t)\langle t\rangle^{-2\delta }+r^{7}(t)\langle t\rangle+r^{5}(t)\langle t\rangle^{1-(1_{-})}\},
			\end{aligned}
			\end{equation}
			where in the last step, we used once again that $r^{2}(t)=\alpha \langle t\rangle^{-1},$ we deduce that
			$$
			\begin{aligned}
			\frac{d}{dt}\{\langle t\rangle^{\alpha}\|\sqrt{\rho}u\|^{2}_{L^{2}}\}&\leq C\mathcal{H}_{0} \langle t\rangle ^{\alpha}\{\langle t\rangle^{-1-2\delta}+\langle t\rangle^{-\frac{5}{2}_{-}}\}\\
			&\leq C\mathcal{H}_{0} \langle t\rangle ^{-1-2\delta+\alpha},
			\end{aligned}
			$$
			and integrating the results over~$[t_{0},t]$~one has                                                                                                             
			\begin{equation}\label{e22}                                                  \|u\|^{2}_{L^{2}}\leq C\mathcal{H}_{0} \langle t\rangle ^{-2\delta}.
			\end{equation}
			This completes the proof of Proposition \ref{Proposition-A.1}.
		\end{proof}
		
		\begin{Corollary}\label{Corollary-5.2}
			Under the assumptions of Proposition \ref{Lemma-u1}, we have
			\begin{equation}\label{d44}
			\begin{aligned}
			\|t^{\delta}u\|^2_{L^{\infty}([t_{0},T];L^{2})}+\|t^{\delta_{-}}\nabla u\|^2_{L^{2}([t_{0},T];L^{2})}\leq C\mathcal{H}_{0},
			\end{aligned}
			\end{equation}
			where $\mathcal{H}_{0}$ given by (\ref{h}).
		\end{Corollary}
		\begin{proof} 
			Multiplying (\ref{d35}) by $t^{2\delta_{-}}$ and then integrating the resulting inequality over $[t_{0},T]$, we get, by applying Proposition \ref{Proposition-A.1}, that
			\begin{equation}\label{d45}
			\begin{aligned}
			&\|t^{\delta_{-}}u\|^2_{L^{\infty}([t_{0},T];L^{2})}+\|t^{\delta_{-}}\nabla u\|^2_{L^{2}([t_{0},T];L^{2})}\\
			\leq& C\|u(t_0)\|_{L^2}+\int^{T}_{t_{0}}t^{(-1)_{-}}\|t^{\delta}u(t)\|^{2}_{L^{2}}dt\leq C\mathcal{H}_{0}.
			\end{aligned}
			\end{equation}
			We completes the proof of Corollary \ref{Corollary-5.2}.
		\end{proof}
	\end{appendix}
	
\end{document}